\documentclass{amsart}

\usepackage{amsmath}
\usepackage{amssymb}
\usepackage{amsthm}
\usepackage[active]{srcltx}
\usepackage{graphicx,color}
\newtheorem{thm}{Theorem}[section]
\newtheorem{prop}[thm]{Proposition}
\newtheorem{lem}[thm]{Lemma}
\newtheorem{cor}[thm]{Corollary}

\newtheorem{procedure}[thm]{Procedure}

\theoremstyle{definition}

\newtheorem{remark}[thm]{Remark}

\newcommand{\sH}{{\mathcal H}}

\newcommand{\sM}{{\mathcal M}}
\newcommand{\sU}{{\mathcal U}}

\newcommand{\sX}{{\mathcal X}}
\newcommand{\sN}{{\mathcal N}}
\newcommand{\sY}{{\mathcal Y}}

\def\a{{\alpha}}
\def\b{\beta}
\def\d{\delta}
\def\de{\Delta}

\def\g{\gamma}
\def\ga{\Gamma}

\def\l{\lambda}
\def\la{\Lambda}

\def\va{\varphi}

\def\om{\Omega}

\def\tht{\Theta}

\def\z{\zeta}
\def\ts{\times}

\def\iy{\infty}
\def\im{{\rm Im\, }}

\def\kr{{\rm Ker\, }}

\def\rank{{\rm rank\, }}

\def\col{{\rm col\, }}

\def\lg{\langle}
\def\rg{\rangle}

\newcommand{\BC}{{\mathbb C}}
\newcommand{\BD}{{\mathbb D}}

\newcommand{\BT}{{\mathbb T}}

\newcommand{\fR}{\mathfrak{R}}

\newcommand{\ands}{\quad\mbox{and}\quad}

\newcommand{\mat}[2]{\ensuremath{\left[\begin{array}{#1}#2\end{array} \right]}}

\newcommand{\inn}[2]{\ensuremath{\langle #1,#2 \rangle}}

\begin{document}

\title[State space formulas for stable rational solutions of a Leech problem]
{State space formulas for stable rational matrix solutions of a Leech problem}

\author[A.E. Frazho]{A.E. Frazho}

\address{%
Department of Aeronautics and Astronautics, Purdue University\\
West Lafayette, IN 47907, USA}

\email{frazho@ecn.purdue.edu}


\author[S. ter Horst]{S. ter Horst}

\address{%
Unit for BMI, North-West University\\
Private Bag X6001-209, Potchefstroom 2520, South Africa}

\email{sanne.terhorst@nwu.ac.za}


\author[M.A. Kaashoek]{M.A. Kaashoek}

\address{%
Department of Mathematics,
VU University Amsterdam\\
De Boelelaan 1081a, 1081 HV Amsterdam, The Netherlands}

\email{m.a.kaashoek@vu.nl}

\thanks{The third author gratefully acknowledge the support of the NWU Mathematics Department at Potchefstroom, SA, during his visit in the Autumn of 2012.}




\date{}

\begin{abstract}
Given stable rational matrix functions $G$ and $K$, a procedure is presented to compute a stable rational matrix solution $X$ to the Leech problem associated with $G$ and $K$, that is, $G(z)X(z)=K(z)$ and $\sup_{|z|\leq 1}\|X(z)\|\leq 1$. The solution is given in the form of a state space realization, where the matrices involved in this realization are computed from state space realizations of the data functions $G$ and $K$.
\end{abstract}

\maketitle

\setcounter{section}{0}
\setcounter{equation}{0}

\section{Introduction}\label{secintro}

Throughout this paper $G$  and $K$ are stable rational complex-valued matrix functions of sizes $m\ts p$ and $m\ts q$, respectively. Here \emph{stable} means that $G$ and $K$ have no poles in the closed unit disc $|z|\leq 1$. In particular, $G$ and $K$ are  matrix-valued $H^\infty$ functions on the open unit disc $\BD$.  For simplicity we write  $G\in \fR  H_{m\ts p}^\iy$  and $K\in {\fR}H_{m\ts q}^\iy$, where $\fR$ stands for rational.  We say that a $p\ts q$ matrix-valued $H^\iy$ function $X$ is a \emph{contractive analytic solution to} $G X = K$ if
\begin{equation}\label{Leech1}
G(z)X(z)= K(z)    \quad (z\in \BD)  \ands \|X\|_\iy=\sup_{z\in\BD}\|X(z)\|\leq 1.
\end{equation}
Leech's theorem (see \cite[page 107]{RR85} or \cite[Section VIII.6]{FF90}) tells us that there exists an $X\in H_{p\ts q}^\iy $ such that \eqref{Leech1} holds if and only if
\begin{equation}
\label{poscond1}T_GT_G^*-T_KT_K^* \ \mbox{is nonnegative}.
\end{equation}
Here $T_G:\ell^2_+(\BC^p)\to\ell^2_+(\BC^m)$ and $T_K:\ell^2_+(\BC^q)\to\ell^2_+(\BC^m)$ are the  (block) Toeplitz operators defined by $G$ and $K$ respectively. The positivity condition \eqref{poscond1} is also equivalent to the requirement that the map
\begin{equation}\label{poskern1}
L(z,\lambda)=\frac{G(\l)G(z)^*-K(\l)K(z)^*}{1-\l\bar{z}}\quad(z,\l\in\BD)
\end{equation}
is  a positive kernel in the sense of Aronszajn \cite{A50}, that is,  (again see \cite[page 107]{RR85}) that for all finite sequences $z_1,\ldots,z_r\in\BD$  and $x_1,\ldots,x_r\in\BC^m$, where $r$ is an arbitrary positive integer, we have
\begin{equation}\label{poskern2}
\sum_{j, k=1}^r   \frac{  \lg \big(G(z_k)G(z_j)^*-K(z_k)K(z_j)^*\big)x_j, x_k\rg}{(1-\bar{z}_j z_k)}  \geq 0.
\end{equation}

The special case of Leech's theorem with $q=m$ and $K$ identically equal to the $m\ts m$ identity matrix $I_{m}$ is part of the corona theorem, which is due to Carlson \cite{Carl62}, for $m=1$, and Fuhrmann \cite{Fuhr68}, for arbitrary $m$. An algorithm to produce rational solutions to the corona problem with $m=1$ and polynomial data functions is given in \cite{Trent07}. For an engineering perspective on the corona problem and its applications in signal processing see \cite{WBP09,WB12} and the references therein.

When $G$ and $K$ are rational, it is known (see \cite{Trent12} or \cite{tH13}) that condition \eqref{poscond1} is also necessary and sufficient for the existence of   stable rational matrix solutions of \eqref{Leech1}. In the present paper we derive a state space formula for a rational matrix solution whose McMillan degree is at most equal to the McMillan degree of $[G\  K]$ starting from state space realizations for $G$ and $K$. Along the way, we obtain a self contained proof of the existence of a rational matrix solution.

The fact that $G$ and $K$ are stable rational matrix functions implies that the function $\begin{bmatrix}G(z) & K(z) \end{bmatrix}$ is also a stable rational matrix function and hence, as is well-known from mathematical systems theory  (see, e.g., Chapter 1 of \cite{CF03} or  Chapter 4 in \cite{BGKR08}), admits a minimal state space realization of the following form:
\begin{equation}
\label{reprGK1}
\begin{bmatrix}
  G(z)  & K(z)
\end{bmatrix}
 = \begin{bmatrix}
  D_1   & D_2
\end{bmatrix}  + z C(I_n -  z A)^{-1}\begin{bmatrix}
  B_1   & B_2
\end{bmatrix} .
\end{equation}
Here $I_n$ is the $n\ts n$ identity matrix, $A$ is a square matrix of order $n$, and $B_1$, $B_2$, $C$, $D_1$ and $D_2$ are matrices of appropriate sizes.  Moreover, $A$ is a \emph{stable} matrix, that is, $A$ has all its eigenvalues  in the open unit disc $\BD$.  In what follows we denote  by $W_{obs}$  the observability operator defined by the pair $\{C,A\}$,  and  for $j=1,2 $ we denote by $P_j$ the controllability  Gramian of the pair $\{A,B_j\}$, that is
\begin{equation}\label{WobsP12}
W_{obs} =  \begin{bmatrix}
C  \\
C A  \\
C A^2 \\
\vdots
\end{bmatrix}\ands P_j=\sum_{\nu=0}^\iy A^\nu B_j B_j^* (A^*)^\nu\quad (j=1,2).
\end{equation}
Note that $W_{obs}$ is an operator mapping $\BC^n$ into $\ell_+^2(\BC^m) $,  and $P_1$ and $P_2$ are $n\ts n$ matrices that satisfy the Stein equations
\[
P_1=AP_1A^*+B_1B_1^*\quad\mbox{and}\quad P_2=AP_2A^*+B_2B_2^*.
\]
Minimality means there exists no realization as in \eqref{reprGK1} with `state operator' $A$ a matrix of smaller size than the one in the given realization.
Our first main result is the following theorem.

\begin{thm}\label{mainthm1} Let  $G\in {\fR}H^\iy_{m\ts p}$ and $K\in {\fR}H^\iy_{m\ts q}$  be given by the minimal realization \eqref{reprGK1}.
Assume that $T_GT_G^* - T_KT_K^* \geq 0$. Then  there exists a function {$F\in {\fR}H^\iy_{m\ts r}$}, for some $r\leq m$,  of the form
\begin{equation}
\label{introF}
F(z)=D_3+zC(I_n-A)^{-1}B_3,
\end{equation}
such that the following  holds:
\begin{itemize}

\item[\textup{(i)}] $T_GT_G^*-T_KT_K^*-T_FT_F^*= W_{obs}(P_3+P_2-P_1)W_{obs}^*$, where $P_3$ is the controllability  Gramian of the pair $\{A, B_3\}$;

\item[\textup{(ii)}]  $P_3+P_2-P_1$ is nonnegative.
\end{itemize}
In particular,  $T_GT_G^*-T_KT_K^*-T_FT_F^*$ is nonnegative and has rank at most $n$, and
\begin{equation}
\label{introGKF}
G(e^{it})G(e^{it})^*-K(e^{it})K(e^{it})^*-F(e^{it})F(e^{it})^*=0 \quad (t\in[0,2\pi]).
\end{equation}
\end{thm}

We see the above theorem as the state space version of the rational matrix analogue of Theorem 0.1 in  \cite{tH13}. Furthermore, to construct the function $F$ in \eqref{introF}  we follow  the method of proof  given in Section 2 of \cite{tH13}, specifying each step in an appropriate state space setting, and using the fact {that}
\begin{equation}
\label{imHankel}
\im H_G+  \im H_K =\im  \begin{bmatrix} H_G& H_K  \end{bmatrix} =\im W_{obs},
\end{equation}
where $H_G$ and $H_K$ are the Hankel operators defined by  $G$ and $K$, respectively.   In the construction of    $F$   an important role is played   by the rational $m\ts m$ matrix  function $R$ defined by
\begin{equation}
\label{defRintro}
R(z)= G(z)G(\bar{z}^{-1})^*-K(z)K(\bar{z}^{-1})^*.
\end{equation}
Using \eqref{poskern2} one sees that   the  positivity condition \eqref{poscond1} implies that $R$ is nonnegative on the unit circle, and hence $R$ admits an outer spectral factor $\Phi$, that is, $\Phi$ is an outer function in $\fR H^\iy_{r\ts m}$, for some $r\leq m$, such that $R(z)=\Phi(\bar{z}^{-1})^*\Phi(z)$. The construction of $F$ is then done in three steps:
\begin{enumerate}
\item Construct a state space  realization for the outer spectral factor $\Phi$.

\item Put
\[
\sM_\Phi=\{f\in \ell_+^2(\BC^r)\mid T_\Phi^*f\in \im W_{obs}\},
\]
which is a backward shift invariant subspace of $\ell_+^2(\BC^r)$, and construct a state space  realization for the 2-sided inner function $\tht$ determined by  $\kr T_\tht^*= \sM_\Phi$.

\item Put $F=\Phi^*\tht$, and compute a state space  realization for $F$.
\end{enumerate}
The explicit constructions of state space realizations for $\Phi$, $\tht$ and $F$ are given in Section \ref{specfact}.

As soon as Theorem \ref{mainthm1} is   proved  we can use  the ``lurking isometry''  approach to Leech's theorem from  Ball-Trent \cite{BT98} to derive stable rational matrix solutions to  the Leech problem \eqref{Leech1}. The next theorem is our second main result.

\begin{thm}\label{mainthm2} Let  $G\in {\fR}H^\iy_{m\ts p}$ and $K\in {\fR}H^\iy_{m\ts q}$ be given by the minimal realization \eqref{reprGK1}, and let $F\in {\fR}H^\iy_{m\ts r}$ be as in \textup{Theorem \ref{mainthm1}}.
Let $Y$ be the solution of the Stein equation
\[
Y = A^* Y A + C^*C,\quad \mbox{that is,}\quad Y=\sum_{\nu=0}^\iy (A^*)^\nu C^*C A^\nu,
\]
set  $\Upsilon=(P_3+P_2-P_1)^{1/2}$, and let
\begin{equation}
\label{intropartiso1}
U = \begin{bmatrix}
\alpha  &  \beta_1 & \beta_2 \\
\gamma  &  \delta_1 & \delta_2 \\
\end{bmatrix}: \begin{bmatrix} \BC^n  \\  \BC^q \\  \BC^r \end{bmatrix}
\to \begin{bmatrix} \BC^n  \\   \BC^p   \end{bmatrix}
\end{equation}
be defined by
 \begin{align}
U&=  \begin{bmatrix}
 \Upsilon Y \Upsilon & \Upsilon Y B_1  \\
B_1^* Y \Upsilon  & D_1^*D_1 + B_1^* Y B_1
\end{bmatrix} ^+ \ts \nonumber \\
&\hspace{.5cm}\ts
\begin{bmatrix}
\Upsilon Y A \Upsilon  & \Upsilon Y B_2    & \Upsilon Y B_3  \\
D_1^* C \Upsilon +B_1^* Y A \Upsilon  & D_1^*D_2 + B_1^* YB_2 &  D_1^*D_3  + B_1^* Y B_3
 \end{bmatrix}. \label{defU1}
 \end{align}
 Here the  superindex  ${}^+$ means that we take the Moore-Penrose generalized inverse of the matrix involved. Then $U$ is a partial isometry and  the following conditions hold:
\begin{itemize}

\item[\textup{(i)}] the function $X$ defined on $\BD$ by
\begin{equation}\label{introX}
X(z) =  \delta_1 + z  \gamma (I - z \alpha)^{-1} \beta_1
\end{equation}
is a $p\ts q$ stable  contractive  rational matrix   solution to the Leech problem  \eqref{Leech1};

\item[\textup{(ii)}] the function $\Psi$ defined on $\BD$ by
\begin{equation}\label{introPsi}
\Psi(z) =  \delta_2 + z  \gamma (I - z \alpha)^{-1} \beta_2
\end{equation}
is a $p\ts r$ stable rational matrix function,  $\|\Psi\|_\infty \leq 1$, and $\Psi$ satisfies the equation $G(z)\Psi(z)=F(z)$.

\end{itemize}
\end{thm}
As we shall see, the proof of  the above theorem uses the fact that  item  (i)   in Theorem \ref{mainthm1} yields the  identity:
\begin{align}
 &\l\bar{z}{\la}(\l){\la} (z)^*+G(\l)G(z)^*=\nonumber  \\
&\hspace{1cm}  ={\la}(\l){\la} (z)^*+ K(\l)K(z)^*+F(\l)F(z)^*\quad (z,\l \in\BD),\label{fundid1}
\end{align}
where  $\la(z)=C(I_n-zA)^{-1}(P_3+P_2-P_1)^{1/2}$. This allows one to construct a partial isometry $U$ such that
 \begin{equation}
\label{propUintro}
\begin{bmatrix}  z \la(z) & G(z)  \end{bmatrix}U=\begin{bmatrix}    \la(z) & K(z) &   F(z)   \end{bmatrix}.
\end{equation}
In fact, we will show that the  matrix $U$ defined by \eqref{defU1} has these properties. Using the partitioning \eqref{intropartiso1}, the  identity  \eqref{propUintro} yields the results mentioned in items (i) and (ii)  of Theorem \ref{mainthm2}.

\begin{remark}\label{Rem}
It can happen (cf.,  \cite[Theorem 3.2]{tH13}) that the $m\ts m$ rational matrix function $R$ defined by \eqref{defRintro} is identically equal to zero. For instance, take
\[
G(z) = \frac{1}{\sqrt{2}}\begin{bmatrix}  1  & 1   \end{bmatrix} \quad \mbox{and}\quad K(z) = z.
\]
If $R$ is identically equal to zero, then items (i) and (ii) in Theorem \ref{mainthm1} hold true with the function $F$ identically equal to zero and $P_3=0$. Furthermore, Theorem \ref{mainthm2} holds with $\BC^r$  being replaced by $\BC^0=\{0\}$ and setting $\Upsilon=(P_2-P_1)^{1/2}$. See Theorem \ref{thabs} below for further details.
\end{remark}

\begin{remark}\label{Rem2} If the rational matrix function $R$ defined by \eqref{defRintro} is not identically equal to zero, Theorem \ref{mainthm1} tells us that one can reduce the problem to the case where $R$ is identically equal to zero without increasing the complexity of the problem. More precisely, Theorem \ref{mainthm1} shows that there exists   $F\in {\fR}H^\iy_{m\ts r}$ of the form \eqref{introF} such that
condition \eqref{poscond1} holds with $\begin{bmatrix}K&F \end{bmatrix}$ in place of $K$,  the realization
  \[
  \begin{bmatrix}
  G(z)  & K(z)&F(z)
\end{bmatrix}
 = \begin{bmatrix}
  D_1   & D_2  & D_3
\end{bmatrix}  + z C(I_n -  z A)^{-1}\begin{bmatrix}
  B_1   & B_2  & B_3
\end{bmatrix}
\]
is minimal, and  the rational matrix function defined  by  \eqref{defRintro}   with $\begin{bmatrix}K&F \end{bmatrix}$ in place of $K$ is identically equal to zero.
\end{remark}

The paper consists of six sections including the present introduction. In the   second section  we construct the function $F$ following  the
 three steps listed above. This is done in a somewhat more general setting, not  using $G$ and $K$, but only  an $m\ts m$ rational matrix function $R$ which  has no pole on the unit circle $\BT$ and whose values on $\BT$ are nonnegative.  In Section \ref{secProof1} we prove Theorem \ref{mainthm1}.  The proof of Theorem \ref{mainthm2} is given in  Section \ref{secProof2}. In Section \ref{secposdef} we specify the results for the case when on the unit circle the values of the function $R$ defined by \eqref{defRintro} are strictly positive.   In the final section we illustrate the main theorems on an example.

\medskip
\noindent\textbf{Some terminology and notation.} We conclude this introduction with some terminology and notation that will be used throughout the paper.   Given a subspace $\sU$ of a Hilbert space $\sY$ we denote by $E_\sU$ the canonical embedding of $\sU$ into $\sY$. Note that $E_\sU^*$ is the orthogonal projection of $\sY$ onto $\sU$ viewed as an operator from $\sY$ to $\sU$.  Thus the orthogonal projection of $\sY$ onto $\sU$ viewed as an operator on $\sY$  is given by $E_\sU E_\sU^*$.  The latter operator will also be denoted by $P_\sU$. For any positive integer $k$ we write $E$ for the canonical embedding of $\BC^k$ onto the first coordinate space of $\ell_+^2(\BC^k)$, that is, $E^*=  \begin{bmatrix} I_k & 0  & 0   &\cdots  \\  \end{bmatrix}$.  Here $\ell^2_+(\BC^k)$ denotes the Hilbert space of unilateral square summable sequences of vectors in $\BC^k$.

Let $T$ be a bounded linear operator from the Hilbert space $\sU$ into the Hilbert space $\sY$, and assume that $T$ has a closed range. Then $T^+$ denotes the \emph{Moore-Penrose generalized inverse} of $T$, that is, $T^+$ is the unique operator from $\sY$ into $\sU$ such that $T^+T=P_{\im T^*}$ and $T T^+=P_{\im T}$. If $T$ is a Hilbert space operator on $\sU$, i.e., from $\sU$ into $\sU$, then $T$ is called nonnegative in case $\inn{Tu}{u}\geq 0$ for all $u\in\sU$, and strictly positive if $T$ is nonnegative and invertible. We will use the notation $T\geq 0$ to indicate that $T$ is nonnegative.

For a rational matrix function  $\om$ we define $\om^*(z)=\om(\bar{z}^{-1})^*$. If  $\om$ has no poles on the unit circle $\BT$, then $\om^*(\z)=\om(\z)^*$ for any $\z\in \BT$.  If $\om$ is a $k\ts l$   matrix function with entries in $L^\iy$ on the unit circle $\BT$, i.e., $\om$ is measurable and essentially bounded on $\BT$, then
$T_\om$ is the {\em Toeplitz operator} defined by
 \begin{equation}\label{defTF}
T_\om = \left[ \begin{array}{cccc}
      \om_0     & \om_{-1}     & \om_{-2}    & \cdots \\
      \om_1     & \om_0        & \om_{-1}    & \cdots \\
      \om_2     & \om_1        & \om_0       & \cdots \\
      \vdots  & \vdots     & \vdots    & \ddots
   \end{array}\right]:\ell_+^2(\BC^l) \to \ell_+^2(\BC^k).
\end{equation}
Here $\ldots, \om_{-1}, \om_0,  \om_1, \ldots$ are  the (block) Fourier coefficients of  $\om$. The function  $\om$ is   in $H_{k\ts l}^\infty$ if and only if $T_\om$ is a (block) lower triangular Toeplitz matrix. By $H_\om$ we denote the block Hankel operator determined by the block Fourier coefficients $\om_1,  \om_2, \ldots$, that is,
\begin{equation}\label{defHF}
H_\om = \left[ \begin{array}{cccc}
      \om_1    & \Omega_2    & \om_3    & \cdots \\
      \om_2     & \om_3        & \om_4    & \cdots \\
      \om_3     & \om_4        & \om_5       & \cdots \\
      \vdots  & \vdots     & \vdots    & \ddots
   \end{array}\right]:\ell_+^2(\BC^l) \to \ell_+^2(\BC^k) .
\end{equation}

Now assume $\om\in \fR H^\infty_{k\ts l}$. In that case, $\om$ admits a
state space realization of the form
\begin{equation}\label{Omreal}
\om(z)=D+zC(I_n-zA)^{-1}B,
\end{equation}
with $A$ a stable $n\ts n$ matrix, and $B$, $C$ and $D$ matrices of appropriate size. The integer $n$ is referred to as the state dimension. The observability operator $W_{obs}$ and controllability operator $W_{con}$ defined by the pairs $\{C,A\}$ and $\{A,B\}$, respectively, are defined by
\[
W_{obs}=\mat{c}{C\\CA  \\C A^2  \\\vdots} :\BC^n\to\ell^2_+(\BC^k),\ \
W_{con}=\mat{c}{B^*\\ B^* A^* \\ B^*  A^{*2}\\\vdots}^*:\ell^2_+(\BC^l)\to\BC^n.
\]
Moreover, the observability Gramian $P$ and controllability Gramian $Q$ are the $n \ts n$ matrices given by
\begin{align*}
P&=W_{obs}^* W_{obs}=\sum_{\nu=0}^\infty A^{*\nu} C^*C A^{\nu},\\
Q&=W_{con} W_{con}^*=\sum_{\nu=0}^\infty A^{\nu} B B^* A^{*\nu}.
\end{align*}
The pair $\{C,A\}$ (or the realization \eqref{Omreal}) is called observable in case $P$ is strictly positive, or equivalently, if $\kr W_{obs}=\{0\}$, and the pair $\{A,B\}$ (or the realization \eqref{Omreal}) is called controllable in case $Q$ is strictly positive, or equivalently, if $\im W_{obs}=\BC^k$. It is well known that the realization \eqref{Omreal} is minimal, i.e., there is no state space realization of $\om$ with smaller state dimension, if and only if the realization \eqref{Omreal} is observable and controllable. Finally, note that, given the realization \eqref{Omreal}, we have $H_\om=W_{obs} W_{con}$, and hence $H_\om H_\om^*=W_{obs} Q W_{obs}^*$.

\setcounter{equation}{0}
\section{State space formulas for the outer spectral factor and related functions}\label{specfact}

In this section $R$ is a non-zero  $m\ts m$ rational matrix function with no pole on the unit circle $\BT$. We assume  that $R(\z)$ is hermitian for each $\z\in \BT$, and hence  $R$ admits a state space realization of the following form:
\begin{equation}\label{defR}
R(z) = zC(I_n - zA)^{-1}\ga +
R_0 + \ga^*(zI_n - A^*)^{-1}C^*.
\end{equation}
Here $I_n$ is the $n\ts n$ identity matrix, and $A$ is a stable $n\ts n$ matrix, i.e., all the eigenvalues of ~$A$ are in the open unit disc $\BD$. In the sequel $W_{obs}$ denotes the observability operator defined by the pair $\{C,A\}$, that is,  $W_{obs}$ is the map from $\BC^n$ into $\ell^2_+(\BC^m)$ given by the first identity in \eqref{WobsP12}.

Throughout this section we shall assume that $R(\z)$ is a nonnegative matrix for each  $\z\in \BT$. At this level of generality we shall carry out the three steps of the procedure outlined in the introduction, leading to the construction of a function $F$ with the properties stated in Theorem \ref{mainthm1}.

\paragraph{Step 1: The outer spectral factor $\Phi$.}
The assumption that $R(\z)\geq 0$ on $\BT$ implies (see \cite[Section 6.8]{RR85}) that the Toeplitz  operator $T_R$ is a nonnegative operator and  $R$ admits an \emph{outer spectral factor} $\Phi$, that is, $\Phi$ is in ${\fR}H_{r\ts m}^\infty$, for some $r\leq m$,  such that
\begin{equation}\label{outs}
  R(z) = \Phi^*(z) \Phi(z)
\end{equation}
and the range of the Toeplitz operator $T_\Phi$ is  a dense set in $\ell_+^2(\BC^r)$. Recall (see the final paragraph of Section \ref{secintro})   that $\om^*(z)=\om(\bar{z}^{-1})^*$ for any rational matrix  function  $\om$. The outer spectral factor $\Phi$ is unique up to a unitary constant operator on the left, that is, if $\Psi$ is another outer function satisfying  $R(z) = \Psi^*(z) \Psi(z)$, then $\Phi(z) = U\Psi(z)$ where $U$ is a constant unitary operator; see \cite{NFBK, FB10} for further details.

The following theorem shows how a state space realization of $\Phi$ can be constructed from the state space realization of $R$. It does not require the pair $\{C,A\}$ to be observable.

\begin{thm}\label{thmR1} Let $R$ be as in \eqref{defR}. Assume that $R(\z)\geq 0$  for each $\z\in \BT$, and let $\Phi \in {\fR}H_{r\ts m}^\infty$ be an outer spectral factor of $R$. Put
\begin{equation}\label{XPhi}
\sX_\Phi=\{x\in \BC^n\mid W_{obs} x\in \im T_\Phi^*\}.
\end{equation}
Then $\sX_\Phi$ is invariant under $A$, the space $\bigvee_{\nu=0}^\iy  A^\nu\ga\BC^m$ is contained in $\sX_\Phi$, and there exists a $r\ts n$ matrix $C_\Phi$ such that
\begin{align}
&\Phi(z)= \Phi(0)+ zC_\Phi(I_n - zA)^{-1}\ga;\label{eqPhi}\\[.3cm]
&W_{obs}x=T_\Phi^* W_{\Phi,\,obs}x, \quad x\in \sX_\Phi. \label{basicid}
\end{align}
Here  $W_{\Phi,\,obs}=\mat{cccc}{C_\Phi^*&A^*C_\Phi^*&A^{*2}C_\Phi^*&\cdots}^*$ which is  the observability operator defined by the pair $\{C_\Phi,A\}$. Moreover, $C_\Phi|_{\sX_\Phi}$ is uniquely determined by \eqref{basicid}. Furthermore, defining  $Q_\Phi$ to be the observability Gramian of the pair $\{C_\Phi, A\}$, that is, $Q_\Phi=\sum_{\nu=0}^\iy (A^\nu)^*C_\Phi^*C_\Phi A^\nu$, we have
\begin{align}
 \Phi(0)^*C_\Phi x&=Cx-\ga^* Q_\Phi  A x\quad (x\in \sX_\Phi), \label{cstar}\\[.3cm]
  \Phi(0)^*\Phi(0)&=R_0-\ga^*Q_\Phi  \ga.\label{r0ss}
\end{align}
In particular, $R_0-\ga^*Q_\Phi  \ga$ is nonnegative.
\end{thm}

Although \eqref{basicid} only determines $C_\Phi$ uniquely on $\sX_\Phi$, the fact that $\sX_\Phi$ is invariant under $A$ and $\Gamma \BC^m\subset \sX_\Phi$ implies that we can define $C_\Phi$ on the orthogonal complement of $\sX_\Phi$ arbitrarily, without violating \eqref{eqPhi}--\eqref{r0ss}.

In Section \ref{secposdef} we shall further specify Theorem \ref{thmR1} for the case when the values of $R$ on the unit circle are strictly positive. As we shall see, in that case $C_\Phi$ is uniquely determined, and hence so is the observability Gramian $Q_\Phi$, and $Q_\Phi$ appears as the stabilizing solution of a certain algebraic {Riccati} equation.

\begin{proof}[\bf Proof of Theorem \ref{thmR1}]
 We split the proof into four parts.

\smallskip\noindent \textsc{Part 1.} In this part we show that $\sX_\Phi$ is invariant under $A$ and that $\sX_\Phi$ contains $\bigvee_{\nu=0}^\iy  A^\nu\ga\BC^m$.

Take $x\in \sX_\Phi$. Then there exists a $f\in \ell_+^2(\BC^r)$ such that $W_{obs}x=T_\Phi^*f$.  Let $S_m$ and $S_r$ be the (block) forward shifts on $\ell_+^2(\BC^m)$ and $\ell_+^2(\BC^r)$, respectively. Since $T_\Phi$ is an analytic Toeplitz operator, $S_r T_\Phi=T_\Phi S_m$, and hence
\[
W_{obs}Ax=S_m^*W_{obs}x=S_m^*T_\Phi^*f=T_\Phi^*S_r^*f \in \im T_\Phi^*.
\]
If follows that  $Ax\in \sX_\Phi$, and thus $\sX_\Phi$ is invariant under $A$.

To prove the second statement, given the invariance of  {$\sX_\Phi$}
under $A$, it suffices to show that $\ga$ maps $\BC^n$ into $\sX_\Phi$.
To  {accomplish}  this, let  $R_n$ and $\Phi_n$, $n=0,1, 2, \ldots$,
be the $n$-th Fourier coefficients of $R$ and $\Phi$, respectively. The fact that $R=\Phi^*\Phi$ implies that
\begin{equation}
\label{Fourcoeff}
R_j=\Phi_0^*\Phi_j +\Phi_1^*\Phi_{j+1}+  \Phi_2^*\Phi_{j+2}+ \cdots, \quad j=0,1, 2, \ldots.
\end{equation}
If follows that
\begin{equation}
\label{WobsGa}
W_{obs}\ga u=\begin{bmatrix}
  R_1 \\
  R_2  \\
  R_3\\
    \vdots \\
  \end{bmatrix}u=T_\Phi^*
  \begin{bmatrix}
  \Phi_1 \\
   \Phi_2  \\
 \Phi_3\\
    \vdots \\
  \end{bmatrix}u, \quad u\in \BC^m.
\end{equation}
This proves that $\ga u\in \sX_\Phi$ for each $u\in \BC^m$.

\smallskip\noindent \textsc{Part 2.} In this part  we define $C_\Phi$, derive  \eqref{basicid}, and prove the uniqueness statement.

Take $x\in \sX_\Phi$. Then there exists  {a} $f\in \ell_+^2(\BC^r)$ such that $T_\Phi^*f=W_{obs}x$. Since $T_\Phi^*$ is one-to-one, the vector $f$  is uniquely determined by $x$, and hence there exists a unique linear map $W$ from $\sX_\Phi$ into $\ell_+^2(\BC^r)$ such that
\begin{equation}
\label{defW}
T_\Phi^*Wx=W_{obs}x \qquad (x\in \sX_\Phi).
\end{equation}
We use $W$ to define $C_\Phi$ as follows:
\begin{equation}
\label{defCPhi}
C_\Phi = E^*WE_{\sX_\Phi}^*:\BC^n\to \BC^r.
\end{equation}
Here $E: \BC^r\to \ell_+^2(\BC^r)$ and $E_{\sX_\Phi}:\sX_\Phi\to \BC^n$  are the embedding operators defined in the final paragraph of Section \ref{secintro}.

Next we prove \eqref{basicid}.  Using the canonical embedding of $\sX_\Phi$ into $\BC^n$ we can rewrite  \eqref{defW}   as  $T_\Phi^*W=W_{obs}E_{\sX_\Phi}$.  Recall that  $S_r T_\Phi=T_\Phi S_m$. Thus
\[
T_\Phi^* S_r^* W =  S_m^* T_\Phi^* W =  S_m^* W_{obs}E_{\sX_\Phi} = W_{obs} A E_{\sX_\Phi} = T_\Phi^*  W A E_{\sX_\Phi}.
\]
The fact that  $T_\Phi^*$ is one-to-one implies that $S_r^* W = W A{E_{\sX_\Phi}}$.
Since  $W$ maps $\sX_\Phi$ into $\ell_+^2(\mathbb{C}^r)$, the operator
$W$ admits a matrix representation of the form:
\[
W =  \begin{bmatrix} Y_0^* & Y_1^*  & Y_2^*   &\cdots  \\  \end{bmatrix}^*  :\sX_\Phi \rightarrow \ell_+^2(\mathbb{C}^r) .
\]
Notice that $E^* W = Y_0$, where $E$ is as in \eqref{defCPhi}. Using $S_r^{* j} W = W A^j{E_{\sX_\Phi}}$ for any integer $j \geq 1$, we have
\[
Y_j = E^*  S_\nu^{* j} W = E^*   W A^j{E_{\sX_\Phi}} = Y_0 A^j{E_{\sX_\Phi}}.
\]
Therefore $W$ admits a representation of the form
\[
\begin{bmatrix}Y_0  \\ Y_0 A{E_{\sX_\Phi}} \\ Y_0 A^2{E_{\sX_\Phi}}  \\  \vdots  \\ \end{bmatrix} :\sX_\Phi   \rightarrow \ell_+^2(\mathbb{C}^r).
\]
Thus by \eqref{defCPhi} we have  $C_\Phi = Y_0E_{\sX_\Phi}^*:\BC^n\to \BC^r$.   Using the fact that $\sX_\Phi$ is an invariant subspace for $A$, we see that
\begin{equation}
\label{eqW2}
Wx=W_{\Phi,\, obs}x \quad  (x\in \sX_\Phi).
\end{equation}
Since $T_\Phi^* W=W_{obs}E_{\sX_\Phi}$, the   identity  \eqref{eqW2} yields \eqref{basicid}. Finally, because $W$ is uniquely determined by \eqref{eqW2}, the operator $Y_0 = C_\Phi|_{\sX_\Phi}$ is  uniquely determined as well.

\smallskip\noindent \textsc{Part 3.} In this part  we prove  \eqref{eqPhi}. From the first part of the proof we know that $\im \ga$ is contained in $\sX_\Phi$. Thus the identity \eqref{basicid} yields $T_{\Phi}^*W_{\Phi,\, obs}\ga=W_{obs}\ga$. But then we can use \eqref{WobsGa} to show that
\[
T_{\Phi}^*W_{\Phi,\, obs}\ga=W_{obs}\ga=T_{\Phi}^* \begin{bmatrix}
  \Phi_1 \\
   \Phi_2  \\
 \Phi_3\\
    \vdots \\
  \end{bmatrix}.
\]
Since $T_{\Phi}^*$ is one to one, we see that  $W_{\Phi,\, obs}\ga= \col[\Phi_j]_{j=0}^\iy$. Hence $C_\Phi A^{j-1}\ga=\Phi_j$ for $j=1,2, \dots$. The latter is equivalent to \eqref{eqPhi}.

\smallskip\noindent \textsc{Part 4.} In this part we prove \eqref{cstar} and \eqref{r0ss}.  To establish \eqref{cstar},  note that $Q_\Phi  =  W_{\Phi,\, obs}^* W_{\Phi,\, obs}$. Using \eqref{basicid}, we see that each $x\in \sX_\Phi$  we have
\begin{align*}
Cx &= E^* W_{obs}x = E^*T_\Phi^* W_{\Phi,\, obs}x\\
& = \begin{bmatrix}
\Phi_0^* & \begin{bmatrix}
\Phi_1^* & \Phi_2^* & \Phi_3^* & \cdots \\
\end{bmatrix} \\
\end{bmatrix} \begin{bmatrix}
                C_\Phi  \\
                W_{\Phi,\, obs} A   \\
              \end{bmatrix}x \\
&= \Phi_0^* C_\Phi x + \Gamma^* W_{\Phi,\, obs}^* W_{\Phi,\, obs} Ax \quad [\mbox{because of \eqref{WobsGa}}]\\
&= \Phi_0^* C_\Phi x+ \Gamma^* Q_\Phi  Ax\qquad (x\in \sX_\Phi).
\end{align*}
This proves  \eqref{cstar}. For  $j=0$ the identity \eqref{Fourcoeff} yields
\begin{align*}
R_0 &=  \Phi_0^* \Phi_0 + \sum_{k=1}^\infty \Phi_k^* \Phi_k\\
&= \Phi_0^* \Phi_0 + \sum_{k=1}^\infty \Gamma^* A^{*k}C_\Phi^* C_\Phi A^k \Gamma
= \Phi_0^* \Phi_0 + \Gamma^* Q_\Phi \Gamma.
\end{align*}
Therefore \eqref{r0ss} holds.
\end{proof}

\begin{cor}\label{C:obscon}
Let $R$ be as in \eqref{defR}. Assume $R(\zeta)\geq0$ for each $\zeta\in\BT$, and let $\Phi$ be an outer spectral factor of $R$ given by the state space realization \eqref{eqPhi}. Define $\sX_\Phi$ as in \eqref{XPhi}. Set $A_\circ=P_{\sX_\Phi}A|_{\sX_\Phi}$ on $\sX_\Phi$ and $C_\circ=C_\Phi|_{\sX_\Phi}$, and let  $W_{\circ,\, obs}$ denote  the observability operator defined by $\{C_\circ,A_\circ\}$. Then:
\begin{itemize}
\item[\textup{(i)}] If the pair $\{C,A\}$ is observable, then the operator $W_{\Phi,\, obs}|_{\sX_\Phi}$ is one-to-one on $\sX_\Phi$,  {$W_{\Phi,\, obs}|_{\sX_\Phi}=W_{\circ,\, obs}$,} and the pair $\{C_\circ,A_\circ\}$ is observable.

\item[\textup{(ii)}] If the pair $\{A,\Gamma\}$ is controllable, then $\sX_\Phi=\BC^m$.
\end{itemize}
In particular, if $\{C,A\}$ is observable and $\{A,\Gamma\}$ is controllable, then the state space realization \eqref{eqPhi} of $\Phi$ is minimal.
\end{cor}

\begin{proof}[\bf Proof]
We start with claim (i). Assume $\{C,A\}$ is an observable pair. Then $W_{obs}$ is one-to-one. Hence, by \eqref{basicid}, we find that $W_{\Phi,\, obs}|_{\sX_\Phi}$ is one-to-one on $\sX_\Phi$. Since $\sX_\Phi$ is invariant under $A$, it follows that $W_{\Phi,\, obs}|_{\sX_\Phi}=W_{\circ,\, obs}$. Hence $W_{\circ,\, obs}$ is one-to-one, and thus the pair $\{C_\circ,A_\circ\}$ is observable.

Claim (ii) follows directly from the fact that $\{A,\Gamma\}$ being controllable is equivalent to $\bigvee_{\nu=0}^\iy  A^\nu\ga\BC^m=\BC^n$, which by the inclusion $\bigvee_{\nu=0}^\iy  A^\nu\ga\BC^m\subset \sX_\Phi$, derived in Theorem \ref{thmR1}, implies $\sX_\Phi=\BC^n$. In particular, in that case $A_\circ=A$ and $C_\circ=C_\Phi$. Thus if $\{C,A\}$ is observable and $\{A,\Gamma\}$ is controllable, then $\{C_\Phi,A\}$ is observable  as well. Hence the realization \eqref{eqPhi} of $\Phi$ is minimal, as claimed.
\end{proof}

\paragraph{Step 2: The two-sided inner function $\tht$.}
Let  $R$ be given by \eqref{defR}. Assume that $R(\z)\geq 0$ for each $\z\in \BT$, and let $\Phi$ be the outer spectral factor of $R$ defined by \eqref{eqPhi}. We define $\sM_\Phi$ to be the subspace of $\ell^2_+(\BC^r)$ given by
\begin{equation}
\label{defM2}
\sM_\Phi=\{f\in \ell_+^2(\BC^r) \mid T_\Phi^* f\in \im W_{obs}\},
\end{equation}
in line with the definition of $\sM_\Phi$ in the second step of the procedure outlined in the introduction.
Note that $\sM_\Phi$ is invariant under the backward shift $S_r^*$, since for each $f\in\sM_\Phi$ we have
\[
T_\Phi^*S_r^* f=S_m^* T_{\Phi}^*f\in S_m^* \im W_{obs} \subset \im W_{obs},
\]
using the fact that $\im W_{obs}$ is invariant under $S_m^*$. By the Beurling-Lax theorem, there exists a two-sided inner function $\tht\in {\fR}H_{r\ts r}^\iy$ such that $\sM_\Phi=\kr T_\tht^*$. In the sequel we shall refer to $\tht$ as the \emph{inner function determined by} $\sM_\Phi$. Before deriving a state space realization for $\tht$, in Proposition \ref{propR2} below, we first prove an alternative formula for the space $\sM_\Phi$.

\begin{lem}\label{L:MPhi}
Let $W_{\Phi,\, obs}$ be the observability   operator defined by the pair $\{C_\Phi,A\}$. Then the space $\sM_\Phi$ in \eqref{defM2} is also given by
\begin{equation}
\label{defM}
\sM_\Phi=W_{\Phi,\, obs}\sX_\Phi.
\end{equation}
If, in addition, the pair $\{C,A\}$ is observable, then $\sM_\Phi=\im W_{\circ,\, obs}$, where $W_{\circ,\, obs}$ is the observability operator defined by the pair $\{C_\circ,A_\circ\}$ given in \textup{Corollary \ref{C:obscon}}.
\end{lem}

\begin{proof}[\bf Proof]
The identity $T_\Phi^* W_{\Phi,obs}x = W_{obs} x$ for $x$ in $\sX_\Phi$  {in \eqref{basicid}} implies that $\sM_\Phi \supset W_{\Phi,obs} \sX_\Phi$. On the other hand, if $f \in  \sM_\Phi$, then $T_\Phi^*f  = W_{obs} x$ for some $x$ in  {$\BC^n$}, and thus, $x$ must be in $\sX_\Phi$. The identities \eqref{defW} and \eqref{eqW2} show that $f= W_{\Phi,obs}x$ and  {hence} $\sM_\Phi \subset W_{\Phi,obs} \sX_\Phi$. Therefore $\sM_\Phi = W_{\Phi,obs} \sX_\Phi$.  {The identity $\sM_\Phi=\im W_{\circ,\, obs}$ for the case that $\{C,A\}$ is observable now follows directly from Corollary \ref{C:obscon}, part (i).}
\end{proof}

\begin{prop}\label{propR2}
Assume the pair $\{C,A\}$ is observable. Let $\tht\in {\fR}H_{r\ts r}^\iy$  be the two-sided inner function determined by $\sM_\Phi$. Then $\tht$  admits a state space realization of the form:
\begin{equation}
\label{realtht}
\tht(z)=D_{\tht}+zC_\Phi(I_n-zA)^{-1}B_{\tht}.
\end{equation}
Here $C_\Phi$  is as in  \eqref{basicid}, and $B_{\tht}$ and $D_{\tht}$ are matrices of sizes $n\ts r$ and $r\ts r$, respectively,  satisfying the following two identities:
\begin{equation}
\label{idunit}
(B_{\tht}^*Q_\Phi A  +D_{\tht}^*C_\Phi)E_{\sX_\Phi}=0 \ands B_{\tht}^*Q_\Phi B_{\tht}+D_{\tht}^*D_{\tht} =I_r.
\end{equation}
Finally, $Q_\Phi$ is the observability Gramian corresponding to the pair $\{C_\Phi, A\}$.
\end{prop}

\begin{proof}[\bf Proof]
To prove the proposition we apply Theorem III.7.2  in \cite{FFGK98}. Define $A_\circ$ and $B_\circ$ as in Corollary \ref{C:obscon} and let $W_{\circ,\, obs}$ and  $Q_\circ$  be the observability operator, respectively observability Gramian, defined by the pair $\{C_\circ, A_\circ\}$.  Note that, by Lemma \ref{L:MPhi},
\[
\im W_{\circ,\, obs}=W_{\Phi,\,obs}\sX_\Phi= \sM_\Phi=\kr T_\tht^*.
\]
Since the pair $\{C , A \}$ is observable, the same holds true for the pair $\{C_\circ, A_\circ\}$, by Corollary \ref{C:obscon}.  Hence $Q_\circ$ is strictly positive, and according to \cite[Theorem III.7.2]{FFGK98}, see also \cite[Lemma 3.2]{KZ99}, there exist linear maps $B_\circ: \BC^r\to\sX_\circ$ and $D_\circ:\BC^r\to \BC^r$ such that
\begin{align}
& \tht(z)=D_\circ+zC_\circ (I_{\sX_\Phi} -zA_\circ)^{-1}B_\circ,  \label{realtht0} \\[.2cm]
&  \begin{bmatrix}
  A_\circ^*    &   C_\circ^* \\
  B_\circ^*    &   D_\circ^*
\end{bmatrix}
\begin{bmatrix}
 Q_\circ    &  0 \\
  0    &  I_r
\end{bmatrix}
\begin{bmatrix}
  A_\circ     &   B_\circ  \\
  C_\circ   &   D_\circ
\end{bmatrix}=
\begin{bmatrix}
 Q_\circ    &  0 \\
  0    &  I_r
\end{bmatrix}.\label{defBD0}
\end{align}

Now put $D_{\tht}=D_\circ$ and  $B_{\tht}=E_{\sX_\Phi}B_\circ$, where $E_{\sX_\Phi}$ is the canonical embedding of $\sX_\Phi$ into $\BC^n$.  Then \eqref{realtht}  and \eqref{idunit} are satisfied. To see this we first note that the definitions of $A_\circ$ and $C_\circ$ yield
\begin{equation}\label{AC1}
E_{\sX_\Phi}A_\circ=A E_{\sX_\Phi} \ands C_\circ =C_\Phi E_{\sX_\Phi}.
\end{equation}
The first identity implies that
\[
E_{\sX_\Phi}(I_{\sX_\Phi}-zA_\circ)=E_{\sX_\Phi}-z E_{\sX_\Phi}A_\circ=
E_{\sX_\Phi}-z AE_{\sX_\Phi}=(I_n-zA)E_{\sX_\Phi}.
\]
This yields $(I_n-zA)^{-1}E_{\sX_\Phi}=E_{\sX_\Phi} (I_{\sX_\Phi}-zA_\circ)^{-1}$. Recall that $B_{\tht}=E_{\sX_\Phi}B_\circ$. It follows that
\[
(I_n-zA)^{-1}B_{\tht}=(I_n-zA)^{-1}E_{\sX_\Phi}B_\circ=E_{\sX_\Phi} (I_{\sX_\Phi}-zA_\circ)^{-1}B_\circ.
\]
Since $D_{\tht}=D_\circ$ and $C_\Phi=C_\circ E_{\sX_\Phi}$ we see that \eqref{realtht0} implies \eqref{realtht}.

Next we prove \eqref{idunit}. To do this note that   \eqref{defBD0} yields the following two identities (see also \cite[Lemma III.7.3]{FFGK98}):
\begin{equation}
\label{idunit0}
B_\circ^*Q_\circ A_\circ  +D_\circ^*C_\circ =0 \ands B_\circ^*Q_\circ B_\circ+D_\circ^*D_\circ =I_r.
\end{equation}
Using $Q_\circ=E_{\sX_\Phi}^*Q_\Phi E_{\sX_\Phi}$,  $B_{\tht}=E_{\sX_\Phi}B_\circ$, and the first identity in \eqref{AC1} we obtain
\[
B_\circ^*Q_\circ A_\circ=B_\circ^*E_{\sX_\Phi}^*Q_\Phi E_{\sX_\Phi} A_\circ=B_{\tht}^*Q_\Phi A E_{\sX_\Phi} .
\]
Similarly, using the second identity in \eqref{AC1}, we get  $D_\circ^*C_\circ =D_{\tht}^*C_\Phi E_{\sX_\Phi}$. It follows that
\begin{align*}
(B_{\tht}^*Q_\Phi A+D_{\tht}^*C_\Phi)E_{\sX_\Phi}&=  B_{\tht}^*Q_\Phi AE_{\sX_\Phi} + D_{\tht}^*C_\Phi E_{\sX_\Phi} \\[.2cm]
    &  =  B_\circ^*Q_\circ A_\circ+D_\circ^*C_\circ =0.
\end{align*}
This proves the first identity in \eqref{idunit}. The second  identity in \eqref{idunit} follows from  $B_{\tht}=E_{\sX_\Phi}B_\circ$, $D_{\tht}=D_\circ$, and $Q_\circ=E_{\sX_\Phi}^*Q_\Phi E_{\sX_\Phi}$. Indeed,
\begin{align*}
B_{\tht}^*Q_\Phi B_{\tht}+D_{\tht}^*D_{\tht} &= B_\circ^*E_{\sX_\Phi}^* Q_\Phi E_{\sX_\Phi}B_\circ+D_\circ^*D_\circ\\[.2cm]
&=B_\circ^*Q_\circ B_\circ +D_\circ^*D_\circ=I_r.\qedhere
\end{align*}
\end{proof}

\begin{lem}\label{leminv} Assume the pair $\{C, A\}$ is observable. In that case the linear map
\begin{equation}
\label{defOm} \om_\Phi=E_{\sX_\Phi}^*Q_\Phi E_{\sX_\Phi}:  \sX_\Phi \to \sX_\Phi \ \mbox{is invertible}.
\end{equation}
Furthermore, the orthogonal projection  of  $\ell_+^2(\BC^r)$ mapping $\ell_+^2(\BC^r)$ onto the finite dimensional space $ \sM_\Phi$  is given by
\begin{equation}
\label{proj1}
P_{\sM_\Phi}=W_{\Phi,\,obs}{\de} W_{\Phi,\,obs}^*,\  \mbox{where ${\de}=E_{\sX_\Phi}\om_\Phi^{-1}E_{\sX_\Phi}^*:\BC^n\to\BC^n$.}
\end{equation}
\end{lem}

\begin{proof}[\bf Proof]
Note that $Q_\Phi=W_{\Phi,\, obs}^*W_{\Phi,\, obs}$. Since by assumption $\{C, A\}$ is observable, Corollary \ref{C:obscon}, part (i), shows that $W_{\Phi, \, obs}$ is one-to-one on $\sX_\Phi$. The latter is equivalent to $\om_\Phi$ being invertible.

Next, let $\la:\sX_\Phi\to \ell_+^2(\BC^r)$ be the map defined by $\la=W_{\Phi, \, obs}E_{\sX_\Phi}$. Then $\la$ is one-to-one and its range is closed  and equals $\sM_\Phi$. Hence the orthogonal projection onto $\sM_\Phi$ is given by $\la^*(\la^*\la)^{-1}\la$ which yields \eqref{proj1}.
\end{proof}

\begin{cor}\label{cormho2}
The linear map ${\de}$ on $\BC^n$ defined in the second part of  \eqref{proj1} is equal to the controllability Gramian of the pair $\{A, B_\tht\}$, where $B_\tht$ is as in \eqref{realtht}.
\end{cor}

\begin{proof}[\bf Proof]
We shall freely use the notation introduced in the proof of Proposition \ref{propR2}. Since $Q_\circ$ is invertible, the identity \eqref{defBD0} implies that
\[
 \begin{bmatrix}
  A_\circ     &   B_\circ  \\
  C_\circ   &   D_\circ
\end{bmatrix}
\begin{bmatrix}
 Q_\circ^{-1}    &  0 \\
  0    &  I_r
\end{bmatrix}
\begin{bmatrix}
  A_\circ^*    &   C_\circ^* \\
  B_\circ^*    &   D_\circ^*
\end{bmatrix}=
\begin{bmatrix}
 Q_\circ^{-1}    &  0 \\
  0    &  I_r
\end{bmatrix}.
\]
In particular, we have $A_\circ  Q_\circ^{-1} A_\circ^*+B_\circ B_\circ^*= Q_\circ^{-1}$. In other words
\begin{equation}
\label{GramBcirc}
 Q_\circ^{-1}-A_\circ  Q_\circ^{-1} A_\circ^*=B_\circ B_\circ^*.
\end{equation}
Now recall that $B_{\tht}=E_{\sX_\Phi}B_\circ$ and $AE_{\sX_\Phi}=A_\circ E_{\sX_\Phi}$. It follows that
\begin{align}
B_{\tht}B_{\tht}^*&=E_{\sX_\Phi}B_\circ B_\circ^*E_{\sX_\Phi}^*   \nonumber \\[.2cm]
    & = E_{\sX_\Phi}Q_\circ^{-1}E_{\sX_\Phi}^*  -E_{\sX_\Phi}A_\circ  Q_\circ^{-1} A_\circ^*E_{\sX_\Phi}^* \nonumber  \\[.2cm]
    &= E_{\sX_\Phi}Q_\circ^{-1}E_{\sX_\Phi}^*- AE_{\sX_\Phi}Q_\circ^{-1}E_{\sX_\Phi}^*A^*.\label{GramBcirc2}
\end{align}
Since $W_{\circ, \, obs}= W_{\Phi, \, obs}E_{\sX_\Phi}$, we see that
\[
Q_\circ=W_{\circ, \, obs}^*W_{\circ, \, obs}= E_{\sX_\Phi}^*W_{\Phi, \, obs}^*W_{\Phi, \, obs}E_{\sX_\Phi}=E_{\sX_\Phi}^*Q_\Phi E_{\sX_\Phi}.
\]
It follows that  $Q_\circ=\om_\Phi$, where $\om_\Phi$ is defined by \eqref{defOm}. But then \eqref{GramBcirc2} and the definition of ${\de}$ in \eqref{proj1} yield
\[
B_{\tht}B_{\tht}^*=E_{\sX_\Phi}Q_\circ^{-1}E_{\sX_\Phi}^*- AE_{\sX_\Phi}Q_\circ^{-1}E_{\sX_\Phi}^*A^*={\de} -A{\de} A^*.
\]
This proves that ${\de}$ is the controllability Gramian of the pair $\{A, B_\tht\}$.
\end{proof}

\begin{lem}\label{L:DSSprep}
Let $\tht\in H^\iy_{r\ts r}$ be the two-sided inner function determined by $\sM_\Phi$. Then $\im H_\Phi\subset \kr T_\tht^*$ and
\begin{equation}\label{canDSScond}
\im H_\Phi = \kr T_\tht^* \quad \Longleftrightarrow\quad \sX_\Phi=\bigvee_{\nu\geq 0} A^\nu \ga \BC^m.
\end{equation}
\end{lem}

\begin{proof}[\bf Proof]
Set $\sX_{con}=\vee_{\nu\geq 0} A^\nu \ga \BC^m$. Note that $\sX_{con}= \im W_{con}$, where $W_{con}$ is the controllability operator defined by the pair $\{A,\ga\}$, and $\sX_{con}\subset\sX_\Phi$, by Theorem \ref{thmR1}. Then
\[
\im H_\Phi=\im W_{\Phi,\,obs} W_{con}=W_{\Phi,\,obs} \sX_{con}\subset W_{\Phi,\,obs} \sX_{\Phi}=\sM_\Phi=\kr T_\tht^*.
\]
The last but one identity follows from \eqref{defM}. Moreover, since $W_{obs}|_{\sX_\circ}$ is one-to-one, by Corollary \ref{C:obscon}, the above inclusion $W_{\Phi,\,obs} \sX_{con}\subset W_{\Phi,\,obs} \sX_{\Phi}$ turns into an identity if and only if $\sX_{con}=\sX_{\Phi}$. Hence \eqref{canDSScond} holds.
\end{proof}

\paragraph{Step 3: The function $F$.}
The final step in the procedure asks for a state space realization for the function $F$ given by $F=\Phi^*\tht$. The following proposition provides such a realization.

\begin{prop}\label{propR3}
Let $\Phi\in {\fR}H_{r\ts m}^\iy$  be the outer spectral factor of the function $R$ given by  \eqref{defR}, and let $\tht\in {\fR}H_{r\ts r}^\iy$ be the two-sided inner function determined by $\sM_\Phi$.  Assume the pair $\{C,A\}$ is observable. Then the function $F=\Phi^*\tht$ belongs to ${\fR}H_{m\ts r}^\iy$, and $F$ admits the following state space realization:
\begin{equation}
\label{realF}
F(z)=  {\Phi(0)^*}D_{\tht} +\ga^*Q_\Phi B_{\tht} +zC(I_n-zA)^{-1}B_{\tht}.
\end{equation}
Here $Q_\Phi$ is the observability Gramian of the pair $\{C_\Phi, A\}$, and $B_{\tht}$ and $D_{\tht}$ are as in \eqref{realtht}. Furthermore,
\begin{equation}
\label{idRF}
T_R=T_FT_F^*+W_{obs}{\de} W_{obs}^*,
\end{equation}
where ${\de} $ is the linear map on $\BC^n $ defined in  the second part of \eqref{proj1} or, equivalently,  ${\de} $ is the controllability Gramian of the pair $\{A, B_\tht \}$.
\end{prop}

\begin{proof}[\bf Proof]
Since $Q_\Phi$ is the observability Gramian of the pair $\{C_\Phi, A\}$, we have, $Q_\Phi -A^* Q_\Phi A=C_\Phi^*C_\Phi$, and hence
\begin{equation}
\label{idCC*}
zC_\Phi^*C_\Phi = (zI_n-A^*)Q_\Phi +A^* Q_\Phi(I_n-zA) \quad (z\in \BC).
\end{equation}
To get \eqref{realF} we  use the state space formulas  \eqref{eqPhi} and \eqref{realtht} which  represent   $\Phi$ and $\tht$, respectively. This yields:
\begin{align*}
F(z) &= \Phi^*(z)\tht(z)  \\
&= \Phi(0)^*D_{\tht}+\ga^* (zI_n-A^*)^{-1}C_\Phi^*D_{\tht}+\\
&\hspace{2cm}+z\Phi(0)^*C_\Phi (I_n-zA)^{-1}B_{\tht}+\a(z),
\end{align*}
where
\[
\a(z)=\ga^* (zI_n-A^*)^{-1}(zC_\Phi^*C_\Phi)(I_n-zA)^{-1}B_{\tht}.
\]
The identity \eqref{idCC*} then shows that
\begin{align*}
\a(z) &= \ga^* (zI_n-A^*)^{-1}\Big[ (zI_n-A^*)Q_\Phi +A^* Q_\Phi(I_n-zA)\Big](I_n-zA)^{-1}B_{\tht}\\
&=\ga^*Q_\Phi(I_n-zA)^{-1}B_{\tht}+\ga^* (zI_n-A^*)^{-1}A^* Q_\Phi B_{\tht}\\
&=\ga^*Q_\Phi B_{\tht}+ z\ga^*Q_\Phi A(I_n-zA)^{-1}B_{\tht}+\ga^* (zI_n-A^*)^{-1}A^* Q_\Phi B_{\tht}.
\end{align*}
It follows that
\begin{align}
F(z) &= \Phi(0)^*D_{\tht}+\ga^*Q_\Phi B_{\tht}+\nonumber\\
&\hspace{1cm}+\ga^* (zI_n-A^*)^{-1}(C_\Phi^*D_{\tht}+A^* Q_\Phi B_{\tht})+\label{2ndterm}\\
&\hspace{1cm}+ z(\Phi(0)^*C_\Phi +\ga^*Q_\Phi A)(I_n-zA)^{-1}B_{\tht}.\label{3rdterm}
\end{align}

Recall that $B_{\tht}=E_{\sX_\Phi}B_\circ$, where $B_\circ$ is as in \eqref{defBD0}.  In particular,  $B_{\tht}$ maps $\BC^r$ into $\sX_\Phi$.  But  $\sX_\Phi$ is invariant under $A$. Therefore $(I_n-zA)^{-1}B_{\tht}$ maps $\BC^r$ into $\sX_\Phi$. In other words,  for each $u\in \BC^r$ the vector $x=(I_n-zA)^{-1}B_{\tht}u$ belongs to $\sX_\Phi$. Hence $(\Phi(0)^*C_\Phi +\ga^*Q_\Phi A)x=Cx$ by \eqref{cstar}, and it follows that  \eqref{3rdterm} is equal to $+zC(I_n-zA)^{-1}B_{\tht}$.

Next we show that the term in \eqref{2ndterm} is zero. To accomplish
 this, note that \eqref{idunit} shows that $(D_{\tht}^*C_\Phi+B_{\tht}^*Q_\Phi A^*)x=0$
 for each $x\in \sX_\Phi$.  Since $\im \ga\subset \sX_\Phi$ and the space $\sX_\Phi$
 is invariant under $A$, we have the   inclusion   $\im (I_n-zA)^{-1}\ga \subset \sX_\Phi$, and thus $(D_{\tht}^*C_\Phi+B_{\tht}^*Q_\Phi A^*)(I_n-zA)^{-1}\ga$ is identically zero.
 Taking the adjoint shows that the term in \eqref{2ndterm} is zero.   Summarizing we see that \eqref{realF} is proved.

It remains to prove \eqref{idRF}. To do this note that $T_\Phi^* T_\tht= T_{\Phi^*\tht}=T_F$. It follows that
\begin{align}
T_R&=T_{\Phi^*\Phi}=T_\Phi^*T_\Phi=T_\Phi^* T_\tht T_\tht^*T_\Phi + T_\Phi^*(I- T_\tht T_\tht^*)T_\Phi  \label{eqTR1a}\\[.2cm]
&=T_F T_F^* + T_\Phi^*(I- T_\tht T_\tht^*)T_\Phi. \label{eqTR1b}
\end{align}
Recall that  $\sM_\Phi=\kr T_\tht^*$, and hence $P_{ \sM_\Phi}$ is the orthogonal projection on  $\ell_+^2(\BC^r)$ mapping  $\ell_+^2(\BC^r)$ onto $\kr T_\tht^*$. Since $\tht$ is inner, $T_\tht$ is an isometry, and hence the   orthogonal projection on  $\ell_+^2(\BC^r)$ mapping  $\ell_+^2(\BC^r)$ onto $\kr T_\tht^*$ is equal to $I-T_\tht T_\tht^*$, that is,
\begin{equation}
\label{ProjMPhi}
P_{\sM_\Phi}=I-T_\tht T_\tht^*.
\end{equation}
The latter identity, together with \eqref{eqTR1b} and \eqref{proj1}, shows that
\begin{equation}
\label{fundeq3}
T_R= T_F  T_F^*+T_\Phi^*P_{\sM_\Phi}T_\Phi=T_F  T_F^*+T_\Phi^*W_{\Phi,\, obs}{\de} W_{\Phi,\, obs}^*T_\Phi.
\end{equation}
According to the definition of ${\de}$ in the second part of  \eqref{proj1} the operator ${\de}$ maps $\sX_\Phi$ into itself and is zero on $\BC^n\ominus \sX_\Phi$. But then \eqref{defW} tells us that  $T_\Phi^*W_{\Phi,\, obs}{\de} W_{\Phi,\, obs}^*T_\Phi=W_{obs}{\de} W_{obs}^*$ which completes the proof of  \eqref{idRF}.
\end{proof}

Let $F=\Phi^*\tht$ be the rational matrix function defined in the preceding proposition. Since $\tht$  is 2-sided inner,   $\tht\tht^*=\tht^*\tht$ is identically equal to the $r\ts r$ identity matrix. It follows that
\begin{equation}
\label{DSS1}
R=\Phi^*\Phi=\Phi^*\tht\tht^*\Phi=FF^* \ands \Phi=\tht F^*.
\end{equation}
The first identity in \eqref{DSS1} shows that  $F$ appears as left spectral factor of $R$. The  second  identity tells  us that $F$ and $\tht$ appear as the factors in a Douglas-Shapiro-Shields factorization of $\Phi$.

Recall, e.g., from \cite{DSS71} or Sections 4.7 and 4.8 in \cite{FB10}, that a {\em Douglas-Shapiro-Shields (DSS) factorization} of a function $\Phi\in H^\iy_{r\ts m}$ is a factorization $\Phi=\tht F^*$ with $\tht$ a two-sided inner function in $H^\iy_{r \ts r}$ and $F$ a function in $H^\iy_{m\ts r}$. A DSS factorization
$\Phi=\tht_c F_c^*$ of $\Phi$ is called {\em canonical} if the only common right inner factor between $\tht_c$ and $F_c$ is a unitary constant $r\ts r$ matrix.   Moreover, any DSS factorization
$\Phi = \Theta F^*$ admits a decomposition of the form $\Theta = \Theta_c \Theta_1$ and $F = F_c \Theta_1$ where $\Phi=\tht_c F_c^*$
is the canonical factorization and $\Theta_1$ is an inner function.
Finally, it is noted that $\Phi = \Theta F^*$  is a canonical factorization if and
only if  $\im H_\Phi=\kr T_\tht^*$.

Now let $\Phi = \Theta F^*$ be our DSS factorization where $\kr  T_\Theta^* = W_{\Phi,obs} \sX_\Phi$.
Then we have
\begin{align*}
W_{\Phi,obs}\sX_\Phi  &= \kr  T_\Theta^* = \kr T_{\Theta_c \Theta_1}^* =
 \kr T_{\Theta_c}^* \oplus (T_{\Theta_c} \kr T_{\Theta_1}^*)\\
& = \im H_\Phi \oplus (T_{\Theta_c} \kr T_{\Theta_1}^*) \supset\im H_\Phi .
\end{align*}
Hence $W_{\Phi,obs}\sX_\Phi = \im H_\Phi$ if and only if $\kr T_{\Theta_1}^* = \{0\}$,
or equivalently, $\Theta_1$ is a unitary constant. In other words, $W_{\Phi,obs}\sX_\Phi = \im H_\Phi$
if and only if $\Phi = \Theta F^*$  is a canonical factorization.
If   the pair $\{C,A\}$ is observable, then $W_{\Phi,obs}$  is one to one. In this case,  the dimension of $\sX_\Phi$ equals
the rank of $H_\Phi$ (or equivalently the McMillan degree of $\Phi$) if and only if $\Phi = \Theta F^*$  is a canonical factorization.
Thus Lemma \ref{L:DSSprep} yields the following result.

\begin{cor}
Let $R$ be as in \eqref{defR} with $\{C,A\}$ observable and $R(\z)\geq 0$ for each $\z\in \BT$.
Then the DSS factorization $\Phi=\tht F^*$ of the outer spectral factor $\Phi$ of $R$, with $\tht$ and $F$ as in \textup{Propositions \ref{propR2}} and \textup{\ref{propR3}}, respectively, is canonical if and only if
\[
\sX_\Phi=\vee_{\nu\geq 0} A^\nu \ga \BC^m.
\]
In particular, $\Phi=\tht F^*$ is canonical in case the pair $\{A,\ga\}$ is controllable.
\end{cor}

We conclude this section with an observation that will be useful in the next section, and which is still valid at the level of generality considered in the present section.

\begin{lem}\label{L:declem}
Set $\sM=\sM_\Phi=W_{\Phi,\, obs}\sX_\Phi$ and $\sN=\im W_{obs}$, and consider the orthogonal direct sum decompositions
\begin{equation}
\label{decom1}
\ell_2^+(\BC^r)=\sM\oplus  \sM^\perp \ands \ell_2^+(\BC^m)=\sN\oplus \sN^\perp.
\end{equation}
Then, with respect to these decompositions, $T_\Phi^*$ has a $2\times 2$ matrix representation of the form
\[
T_\Phi^*= \begin{bmatrix}
T_{11}^* & T_{21}^* \\[.2cm]
0    & T_{22}^*
\end{bmatrix}:
\begin{bmatrix} \sM \\[.2cm] \sM^\perp \end{bmatrix}
 \to
 \begin{bmatrix}\sN \\[.2cm] \sN^\perp \end{bmatrix}
\]
with $T_{22}^*$ one-to-one. Furthermore, the function $F\in H^\iy_{m\ts r}$ defined in \textup{Proposition \ref{propR3}} satisfies
\begin{align}
T_FT_F^*&=\begin{bmatrix} I_{\sN}   &T_{21}^*\\[.2cm] 0   &T_{22}^*\end{bmatrix}
\begin{bmatrix} 0 &0\\[.2cm] 0& I_{\sM^\perp}\end{bmatrix}
\begin{bmatrix} I_{\sN}  &0\\[.2cm]  T_{21}   & T_{22}\end{bmatrix}\label{decomTR1b},
\end{align}
and the third factor in the right hand side of \eqref{decomTR1b}, that is,
\begin{equation}
\label{outerprop1}
\begin{bmatrix} I_{\sN}   &0\\[.2cm]  T_{21}   & T_{22}\end{bmatrix}:
 \begin{bmatrix}\sN \\[.2cm] \sN^\perp \end{bmatrix}
 \to
\begin{bmatrix}\sN \\[.2cm] \sM^\perp \end{bmatrix}
\end{equation}
has dense range. Here $I_{\sN}$ and $I_{\sM^\perp}$ denote the identity operators on $\sN$, respectively $\sM^\perp$.
\end{lem}

\begin{proof}[\bf Proof.]
The form of the $2\ts 2$ matrix representation of $T_\Phi^*$ is obvious from the fact that $T_\Phi^*$ maps $\sM$ into $\sN$. The formula for $T_FT_F^*$ follows from
\begin{align*}
T_FT_F^*&=T_\Phi^*(I -P_{\sM})T_\Phi=
\begin{bmatrix} T_{21}^*T_{21}   &T_{21}^*T_{22}\\[.2cm] T_{22}^*T_{21}   &T_{22}^*T_{22}\end{bmatrix}\\[.2cm]
&=\begin{bmatrix} T_{21} ^*\\ T_{22}^*\end{bmatrix}\begin{bmatrix} T_{21} & T_{22} \end{bmatrix}\\[.2cm]
&= \begin{bmatrix} I_{\sN}   &T_{21}^*\\[.2cm] 0   &T_{22}^*\end{bmatrix}
\begin{bmatrix} 0 &0\\[.2cm] 0& I_{\sM^\perp}\end{bmatrix}
\begin{bmatrix} I_{\sN}  &0\\[.2cm]  T_{21}   & T_{22}\end{bmatrix}.
\end{align*}
We prove that $T_{22}^*$ is one-to-one. Note that the fact the \eqref{outerprop1} has dense range is a direct consequence of this. To see that $T_{22}^*$ is one-to-one, take $f\in \sM^\perp$ and assume that $T_{22}^*f=0$. This implies that  $T_\Phi^*f\in \sN=\im W_{obs}$. But then \eqref{defM2} tells us that $f\in \sM_\Phi=\sM$. Thus $f\in \sM\cap \sM^\perp$, and $f$ must be zero. Therefore  $T_{22}^*$ is one-to-one, as claimed.
\end{proof}

\setcounter{equation}{0}
\section{Proof of Theorem \ref{mainthm1}}\label{secProof1}

Throughout this section  $G$ and $K$  are stable rational matrix functions, $G\in {\fR}H^\iy_{m\ts p}$ and $K\in {\fR}H^\iy_{m\ts q}$, and we assume that  $\begin{bmatrix}G(z) & K(z) \end{bmatrix}$ is given by a stable  state space realization of the following form:
\begin{equation}
\label{reprGK2}
\begin{bmatrix}
  G(z)  & K(z)
\end{bmatrix}
 = \begin{bmatrix}
  D_1   & D_2
\end{bmatrix}  + z C(I_n -  z A)^{-1}\begin{bmatrix}
  B_1   & B_2
\end{bmatrix}.
\end{equation}
In particular, $A$ is a stable matrix. Note that \eqref{reprGK2} is equivalent to the following two realizations:
\begin{align}
G(z) &=  D_1  + z C(I_n -  z A)^{-1}B_1,\label{reprG1} \\[.2cm]
K(z) &=  D_2  + z C(I_n -  z A)^{-1}B_2.\label{reprK1}
\end{align}
The following lemma will allow us to apply the results of the previous section.

\begin{lem}\label{lemeqR}
Let $G\in {\fR}H^\iy_{m\ts p}$ and $K\in {\fR}H^\iy_{m\ts q}$, and put
\begin{equation}
\label{RGK}
R(z)=G(z)G^*(z)  -K(z)K^*(z) .
\end{equation}
Assume \eqref{reprG1} and \eqref{reprK1} are  stable realizations, and  let $P_j$ be the controllability Gramian for the pair $\{A,B_j\}$ for $j=1,2$. Then $R$ is an ${m\ts m}$ rational matrix function with no pole on $\BT$,   and  $R$  admits the following state space realizations:
\begin{equation}\label{defR2}
R(z) = zC(I_n - zA)^{-1}\ga +
R_0 + \ga^*(zI_n - A^*)^{-1}C^*,
\end{equation}
where
\begin{align}
R_0 &= D_1D_1^* -D_2D_2^* +C(P_1 - P_2)C^*,\label{defR0}\\[.2cm]
 \Gamma & = B_1D_1^* -B_2D_2^* + A(P_1 - P_2)C^*.\label{defga}
\end{align}
\end{lem}

\begin{proof}[\bf Proof]
The result follows from Lemma 3.1 in \cite{FKR2a-10}. Indeed,  applying the latter lemma to $GG^*$  yields
\begin{align*}
G(z)G^*(z)& =z C(I_n-zA)^{-1}\Gamma_1+D_1 D_1^* + C P_1 C^*  \\
&\hspace{1cm}+ \ga_1^*(zI_n-A^*)^{-1}C^*, \ \mbox{where  {$\ga_1= D_1B_1^*+  A P_1 C^*$}}.
\end{align*}
In a similar way  one obtains
\begin{align*}
K(z)K^*(z)& =z C(I_n-zA)^{-1}\Gamma_2+D_2 D_2^* + C P_2 C^* +    \\
&\hspace{1cm}+\ga_2^*(zI_n-A^*)^{-1}C^* \ \mbox{where  $\ga_2 = B_2 D_2^* + A P_2 C^*$}.
\end{align*}
Together these two realizations yield \eqref{defR2}.
\end{proof}

\begin{thm}\label{thabs} Let \eqref{reprGK2} be a  realization  of $\begin{bmatrix} G & K \\ \end{bmatrix}$ which is  minimal, that is, both observable and controllable. Then $T_G T_G^* - T_K T_K^*$ is nonnegative if and only if the following two  conditions hold:
\begin{itemize}
\item[\textup{(i)}] The rational matrix function $R$ defined by \eqref{RGK} has nonnegative values on $\BT$ or, equivalently,   $R$  has an outer spectral factor $\Phi$  belonging to  ${\fR}H_{r\ts m}^\infty$ for some $r\leq m$.
\item[\textup{(ii)}] The operator  ${\de}+P_2-P_1$ is nonnegative. Here  $P_1$ and $P_2$ are  the controllability Gramians corresponding to the pairs $\{A,B_1\}$ and  $\{A,B_2\}$, respectively, and ${\de} $ is the linear map defined by the second part of  \eqref{proj1}.
\end{itemize}
In the special case when the function $R$ defined by \eqref{RGK} is identically zero item $(i)$ is automatically fulfilled $($with $r=0)$  and item $(ii)$ holds with $\de=0$.
\end{thm}

\begin{proof}[\bf Proof]
As noted in the introduction the condition $T_G T_G^* - T_K T_K^*$ is nonnegative  implies that the function $R$ defined  in \eqref{RGK} is nonnegative on $\BT$, or equivalently, $R$ has an outer spectral factor, $\Phi$ say, which belongs to  ${\fR}H_{r\ts m}^\infty$. Therefore in what follows  we shall assume that  condition (i) is fulfilled. Since we assume that  (i) holds, it remains to prove that $T_G T_G^* - T_K T_K^*\geq 0$  if and only if   ${\de}+ P_2-P_1\geq 0$.

A classical identity for Toeplitz and Hankel operators  (see, e.g.,  \cite[Proposition 2.14]{BS90}, or  \cite[Section XXIII.4]{GGK2}) yields
\begin{equation} \label{ToeplHankel}
 T_{GG^*} = T_G T_G^*+H_G H_G^* \ands   T_{KK^*} = T_K T_K^*+H_K  H_K^*.
\end{equation}
Since $R=GG^*-KK^*$, we have $T_R=  T_{GG^*}-T_{KK^*}$. Using the two identities in \eqref{ToeplHankel} we see that
\begin{equation} \label{fundid3a}
T_G T_G^* - T_K T_K^*=T_R-(H_G H_G^*-H_K H_K^*).
\end{equation}
Recall that $H_G H_G^*=W_{obs} P_1 W_{obs}^*$ and $H_K H_K^*=W_{obs} P_2 W_{obs}^*$. Thus \eqref{fundid3a} can be rewritten as
\begin{equation}
\label{fundid3ab}
T_G T_G^* - T_K T_K^*=T_R  +W_{obs}\big(P_2-P_1\big)W_{obs}^*.
\end{equation}

We shall first attend to the case where $R$ is identically zero. As observed in Remark \ref{Rem}, this is not a trivial case. When $R$ is identically zero, identity \eqref{fundid3ab} implies that $T_G T_G^* - T_K T_K^* \geq 0$ if and only if $W_{obs}\big(P_2-P_1\big)W_{obs}^*$ is nonnegative.  But $W_{obs}$ is one-to-one and hence $W_{obs}^*$ has dense range. It follows that
\[
W_{obs}\big(P_2-P_1\big)W_{obs}^*\geq 0\quad   \Longleftrightarrow \quad  P_2-P_1 \geq 0.
\]
This proves the theorem for the case when $R\equiv 0$.

Next assume  that $R$   is not identically zero.  This allows us to apply the results of the previous section. Using \eqref{idRF},  the identity   \eqref{fundid3ab} can be rewritten as
\begin{equation}
\label{fundid3b}
T_G T_G^* - T_K T_K^*=T_FT_F^* +W_{obs}\big({\de} +P_2-P_1\big)W_{obs}^*.
\end{equation}
Now put $\sN=\im W_{obs}\subset\ell^2_+(\BC^m)$, and set $\sN^\perp=\ell_2^+(\BC^m)\ominus \sN$, i.e., $\ell_2^+(\BC^m)=\sN\oplus \sN^\perp$. Let $E_{\sN}$ be the canonical embedding of $\sN$ into $\sN\oplus \sN^\perp$. Then
 \begin{align*}
    &W_{obs}\big({\de} +P_2-P_1\big)W_{obs}^*=    \\
    &\hspace{1cm}  = \begin{bmatrix}
    E_{\sN}^*W_{obs}\big({\de} +P_2-P_1\big)W_{obs}^*E_{\sN} &   0 \\
     0 &  0
\end{bmatrix}:  \begin{bmatrix}\sN \\[.2cm] \sN^\perp \end{bmatrix}\to  \begin{bmatrix}\sN \\[.2cm] \sN^\perp \end{bmatrix}.
\end{align*}
Using  the previous identity and  \eqref{decomTR1b} we obtain
\begin{align*}
&T_G T_G^* - T_K T_K^* =   \\
    &\hspace{1cm}=  \begin{bmatrix} I_{\sN}   &T_{21}^*\\[.2cm] 0   &T_{22}^*\end{bmatrix}
\begin{bmatrix}E_{\sN}^*W_{obs}\big({\de} +P_2-P_1\big)W_{obs}^*E_{\sN}^*&0\\[.2cm] 0&I_{\sM^\perp}\end{bmatrix}
\begin{bmatrix}  I_{\sN}    &0\\[.2cm]  T_{21}   & T_{22}\end{bmatrix}.
\end{align*}
By Lemma \ref{L:declem} the third factor on the right hand side has dense range, and consequently, the first factor on the right hand side has a trivial kernel. It follows that
\begin{align}
T_G T_G^* - T_K T_K^* \geq 0 &\Longleftrightarrow  W_{obs}\big({\de} +P_2-P_1\big)W_{obs}^*\geq 0\label{equiv1}\\[.2cm]
 &\Longleftrightarrow  {\de} +P_2-P_1\geq 0.\label{equiv2}
\end{align}
The second equivalence follows from the fact that   $W_{obs}^*$ has   dense range.  We conclude  (assuming item (i) holds)  that  the operator $T_G T_G^* - T_K T_K^*$ is nonnegative if and only if  item (ii) is satisfied.
\end{proof}

\begin{proof}[\bf Proof of Theorem \ref{mainthm1}]
Assume that  $\begin{bmatrix} G & K \\ \end{bmatrix}$  is given by \eqref{reprGK2} and that  the right hand side of  \eqref{reprGK2} is a minimal realization. Define $R$ by \eqref{defR2}, where $R_0$ and $\ga$ are given by \eqref{defR0} and \eqref{defga}, respectively. The fact that $T_GT_G^*-T_K T_K^*$ is nonnegative implies that $R$ admits an outer spectral factorization, $\Phi$ say, which belongs to ${\fR}H_{r\ts m}^\iy$. Using this $\Phi$, one constructs $\tht$ and $F$ as in Section \ref{specfact}. We claim that $F$ has the desired properties. Indeed, \eqref{realF} shows that $F$ is of the form \eqref{introF} with
\begin{equation}\label{B3D3}
B_3=B_\tht \ands D_3=\Phi(0)^* D_\Theta + \ga^*Q_\Phi B_\tht.
\end{equation}
Furthermore, by Corollary \ref{cormho2}, in this case the controllability Gramian of the pair $\{A, B_3\}=\{A, B_\tht\}$ is equal to the matrix ${\de}$ in \eqref{proj1}. In other words
\[
P_3={\de} \quad\mbox{and}\quad{\de}+P_2-P_1=P_3+P_2-P_1.
\]
But then \eqref{fundid3b} shows that item (i) in Theorem \ref{mainthm1} is satisfied. Finally, the fact that $T_GT_G^*-T_KT_K^*$ is nonnegative implies that  ${\de}+P_2-P_1=P_3+P_2-P_1$ is nonnegative, which proves item (ii).

It remains to prove the final statements in Theorem \ref{mainthm1}. From \eqref{fundid3b} it follows that
\[
\rank (T_G T_G^* - T_K T_K^*-T_FT_F^*)= \rank W_{obs}\big(P_3 +P_2-P_1\big)W_{obs}^*\leq n.
\]
Since $G$, $H$, and $F$ are rational matrix functions, the corresponding Hankel operators have finite rank. Hence, using
\[
T_{\{GG^*-KK^*-FF^*\}}=(T_G T_G^* - T_K T_K^*-T_FT_F^*)-(H_G H_G^* - H_K H_K^*-H_FH_F^*),
\]
it follows that $\rank (T_G T_G^* - T_K T_K^*-T_FT_F^*)$ is finite, implies that the rank of  the Toeplitz  operator $T_{GG^*-KK^*-FF^*}$ is finite. This can only happen when the function $GG^*-KK^*-FF^*$ is zero; cf., \cite[Theorem 3.2]{tH13}.
\end{proof}

\setcounter{equation}{0}
\section{Proof of Theorem \ref{mainthm2}}\label{secProof2}

Let $G\in {\fR}H_{m\ts p}^\iy$  and $K\in {\fR}H_{m\ts q}^\iy$ be stable rational matrix functions, and assume that  $\begin{bmatrix}G(z) & K(z) \end{bmatrix}$ is given by the   minimal   realization \eqref{reprGK1}.  Furthermore, assume that the positivity condition \eqref{poscond1} is satisfied.  Then,    by Theorem \ref{mainthm1} there exists {a} $F\in {\fR}H_{m\ts r}^\iy$, for some $r\leq m$, such that $F$ admits a realization of the form  \eqref{introF} and conditions (i), (ii) in Theorem \ref{mainthm1} are satisfied.

\begin{lem}\label{lempos2}
Condition  \textup{(i)} in \textup{Theorem \ref{mainthm1}} implies that \eqref{fundid1} holds.
\end{lem}

\begin{proof}[\bf Proof]
For each $z$ in the  open unit disc $\mathbb{D}$, let $\varphi_z$ be the operator defined by
\begin{equation}\label{defvaz}
\va _z = \begin{bmatrix}   I_\upsilon & z I_\upsilon &  z^2 I_\upsilon &  \cdots  \end{bmatrix}^*:\BC^\upsilon
\to \ell_+^2(\BC^\upsilon).
\end{equation}
Here $\upsilon$ is an arbitrary positive integer, the value of which will be clear from the context.

Notice that
\begin{align*}
& T_G^* \va_z =\va_z G(z)^*, \quad T_K^* \va_z =\va_z K(z)^*, \quad T_F^* \va_z =\va_z F(z)^*, \\[.2cm]
&\hspace{2cm}\va_z^* W_{obs} = C(I - z A)^{-1},\quad \va_\l^*\va_z=\frac{1}{1-\l \bar{z}}I.
\end{align*}
It follows that  for each  $z$ and $\l$  in $\BD$ we have
\begin{align}
&\va _\l ^* (T_G T_G^* - T_K T_K^*- T_F T_F^*)\va_z=\nonumber \\[.2cm]
&\hspace{3cm} =\frac{G(\l )G(z)^*-K(\l )K(z)^*-F(\l )F(z)^*}{1- \lambda \bar{z}},\label{reproid}\\[.2cm]
& \va _\l ^*  W_{obs}(P_3+P_2-P_1)W_{obs}^*\va_z=\nonumber \\[.2cm]
&\hspace{3cm}=C( I_n - \lambda A)^{-1}(P_3+P_2-P_1)( I_n - \bar{z} A^*)^{-1}C^*.\nonumber
\end{align}
But then condition  (i) in  Theorem \ref{mainthm1}  implies that
\begin{align*}
&G(\l )G(z)^*-K(\l )K(z)^*-F(\l )F(z)^*=\\[.2cm]
&=(1-\l\bar{z})C( I_n - \lambda A)^{-1}(P_3+P_2-P_1)( I_n - \bar{z} A^*)^{-1}C^*  \quad (z, \l\in \BD).
\end{align*}
Recall that  $\la(z)=C(I_n-zA)^{-1}(P_3+P_2-P_1)^{1/2}$. Hence  the preceding identity is just the same as the identity  \eqref{fundid1}.
\end{proof}

\begin{proof}[\bf Proof of Theorem \ref{mainthm2}]
Recall that
\begin{align*}
G(z) &=  D_1  + z C(I_n -  z A)^{-1}B_1, \quad K(z) =  D_2  + z C(I_n -  z A)^{-1}B_2,  \\[.2cm]
F(z) &=  D_3  + z C(I_n -  z A)^{-1}B_3,   \\[.2cm]
\la(z)&=C(I_n-zA)^{-1}\Upsilon,\  \mbox{where $\Upsilon=(P_3+P_2-P_1)^{1/2}$}.
\end{align*}
Next put
\[
M(z) = \begin{bmatrix}  z \la(z) & G(z)  \end{bmatrix},\quad
 N(z) = \begin{bmatrix}    \la(z) & K(z)&    F(z)   \end{bmatrix}.
\]
Using the state space  realizations of $G$, $K$, $F$, and $\la$ given above we see   that $M$ and $N$ admit the following realizations:
\begin{align*}
M(z)&=
 \begin{bmatrix}  0 & D_1 \end{bmatrix} + z C(I_n - z A)^{-1}\begin{bmatrix}
 \Upsilon  & B_1 \end{bmatrix},\\[.2cm]
 N(z) &=   \begin{bmatrix}  C \Upsilon & D_2& D_3   \end{bmatrix}
 + z C (I_n-z A)^{-1}
 \begin{bmatrix} A \Upsilon  &  B_2 & B_3  \end{bmatrix}.
\end{align*}
Furthermore, the   identity  \eqref{fundid1} tells us that
\[
M(\l)M(z)^*= N(\l)N(z)^* \qquad (z, \l \in \BD).
\]
This allows us to apply Lemma \ref{lemrep} below.  It follows that  the linear operator mapping  $ \BC^n  \oplus \BC^q  \oplus \BC^r$ into $\BC^n \oplus \BC^p $ defined by \eqref{defU1} is a partial isometry and   $M(z) U = N(z)$.

Now partition $U$ as in \eqref{intropartiso1}.  Then $M(z) U = N(z)$ is equivalent to
\begin{align*}
&z\la(z)\a+G(z)\g=\la(z), \\
&z\la(z)\b_1+G(z)\d_1=K(z),  \\
&z\la(z)\b_2+G(z)\d_2=F(z).
\end{align*}
The first identity implies that $\la(z)=G(z)\g (I-z\a)^{-1}$.  Using this expressing for $\la(z)$ in the other two identities  yields
\begin{equation}
\label{idsGKF2}
G(z)\big(\d_1+z\g (I-z\a)^{-1}\big)\b_1=F(z) \ands G(z)\big(\d_2+z\g (I-z\a)^{-1}\big)\b_2.
\end{equation}
Since  $U$ is a contraction,  it follows from the bounded real lemma in systems theory or the Sz.-Nagy-Foias model theory in operator theory (see also Theorem 5.2 in \cite{Ando90})  that the matrix function   $\begin{bmatrix} X &\Psi \end{bmatrix}$, with $X$ and $\Psi$ defined as in \eqref{introX} and \eqref{introPsi}, respectively, satisfies $\|\begin{bmatrix} X &\Psi \end{bmatrix}\|_\iy\leq 1$, in particular, $X$  is a rational contractive function on $\BD$.  Furthermore, the first identity in  \eqref{idsGKF2}  implies that $X$ satisfies the Leech equation $GX=K$.  In the same way, using the second   identity in \eqref{idsGKF2}, one shows that the function $\Psi$ in \eqref{introPsi} has the desired properties.
\end{proof}

In the next lemma $M$ and $N$  are stable rational matrix functions, $M\in {\fR}H_{m\ts k}^\iy$ and $N\in {\fR}H_{m\ts \ell}^\iy$. We assume  that  $M$ and $N$ are given by the stable realizations:
\begin{align}
    M(z)&= D_M+zC(I_n-zA)^{-1}B_M,    \label{defM6}\\
     N(z)&= D_N+zC(I_n-zA)^{-1}B_N. \label{defN6}
\end{align}
In particular, $A$ is stable.

\begin{lem}\label{lemrep} Let $M\in {\fR}H_{m\ts k}^\iy$ and $N\in {\fR}H_{m\ts \ell}^\iy$ be given by the stable realizations \eqref{defM6} and \eqref{defN6}, respectively, and let $W_{obs}$ be the observability  operator defined by the pair  $\{C, A\}$. Put $Y=W_{obs}^*W_{obs}$. If
\begin{equation}
\label{condMN6a}
 {M(\l)M(z)^* = N(\l)N(z)^*}
\qquad (z, \l \in \mathbb{D}),
\end{equation}
then the $k \ts \ell$ matrix $U=(D_M^*D_M +B_M^*YB_M )^+(D_M^*D_N +B_M^*YB_N)$ is a partial isometry and    $M (z)U = N(z)$ for all $z$ in $\mathbb{D}$.
\end{lem}

\begin{proof}[\bf Proof]
Let $\om_M$ and $\om_N$ be the operators defined by
\[
\om_M = \begin{bmatrix}   D_M \\ W_{obs}B_M  \end{bmatrix}:\BC^k \to\ell_+^2(\BC^m),
\quad
\om_N = \begin{bmatrix}   D_N \\ W_{obs}B_N  \end{bmatrix}:\BC^\ell  \to\ell_+^2(\BC^m).
\]
For each $z$ in the  open unit disc $\mathbb{D}$, let $\varphi_z$ be the operator defined by \eqref{defvaz}. Then $M(z)^* = \om_M^*  \varphi_z$ and $N(z)^* = \om_N^*  \varphi_z$ for all $z$ in  $\mathbb{D}$.  Thus for $\l $ and $z$ in $\BD$, with
$v$ and $w$ in $\BC^m$, we have
\begin{align*}
\lg\om_M \om_M^* \varphi_\l  w,  \varphi_z v\rg  &= \lg\om_M^* \varphi_\l  w,  \om_M^*\varphi_z v\rg  \\
&= \lg M(\l )^* w, M(z)^* v\rg  = \lg w, M(\l ) M(z)^* v\rg
\end{align*}
and
\begin{align*}
\lg\om_N \om_N^* \varphi_\l  w,  \varphi_z v\rg  &= \lg\om_N^* \varphi_\l  w,  \om_N^*\varphi_z v\rg  \\
&= \lg N(\l )^* w, N(z)^* v\rg  = \lg w, N(\l ) N(z)^* v\rg.
\end{align*}
Because $\{\varphi_z \mathbb{C}^m \mid z \in \mathbb{D}\}$ spans a dense
set in $\ell_+^2(\mathbb{C}^m)$, we see that condition \eqref{condMN6a} implies that
\begin{equation}
\label{condMN6b}
\om_M\om_M^*=\om_N\om_N^*.
\end{equation}
It follows that there exists a unique partial isometry $U$ mapping $\BC^\ell$ into  $\BC^k$ with initial space $\im \om_N^*$ and final space $\im \om_M$  such that
$\om_M U=\om_N$. In fact this unique isometry $U$ is given by   $(\om_M^*\om_M )^+\om_M^*\om_N$, where $(\om_M^*\om_M )^+$ stands for the Moore-Penrose inverse of the finite dimensional selfadjoint  operator $\om_M^*\om_M$.

Finally, using $Y=W_{obs}^*W_{obs}$ and the definitions  of $\om_M$ and $\om_N$ in the beginning of the proof, we obtain
\begin{align*}
\om_M^*\om_M&=
\begin{bmatrix} D_M^*&B_M^*W_{obs}^*\end{bmatrix}
 \begin{bmatrix}   D_M \\ W_{obs}B_M  \end{bmatrix}\\
& =D_M^*D_M +B_M^*W_{obs}^*W_{obs}B_M =D_M^*D_M +B_M^*YB_M.
\end{align*}
and
\begin{align*}
\om_M^*\om_N&=
\begin{bmatrix} D_M^*&B_M^*W_{obs}^*\end{bmatrix}
 \begin{bmatrix}   D_N \\ W_{obs}B_N  \end{bmatrix}\\
& =D_M^*D_N+B_M^*W_{obs}^*W_{obs}B_N=D_M^*D_N +B_M^*YB_N.
\end{align*}
Thus $U=(D_M^*D_M +B_M^*YB_M)^+(D_M^*D_N +B_M^*YB_N)$ as desired.
\end{proof}

\setcounter{equation}{0}
\section{The strictly positive case}\label{secposdef}

We begin by  specifying Theorem \ref{thmR1} for the case when the values of $R$  on the unit circle are strictly positive. If $\Xi$ is an invertible operator on a Hilbert space, then $\Xi^{-*}$ denoted the adjoint of $\Xi^{-1}$.

\begin{prop}\label{propinvTPhi} Let $R$ be as in \eqref{defR}. Assume that $R(\z)$  is strictly positive for each $\z\in \BT$, and let $\Phi \in {\fR}H_{r\ts m}^\infty$ be an outer spectral factor of $R$, as in \textup{Theorem \ref{thmR1}}. Then   $T_\Phi$ is invertible, and
\begin{equation}
\label{invertout}
\sX_\Phi = \BC^n, \quad  C_\Phi = E^* T_\Phi^{-*}W_{obs}, \quad
W_{\Phi, \,obs}=T_\Phi^{-*}W_{obs}.
\end{equation}
Here $E$ is the embedding of $\BC^r$ onto the first coordinate space of $\ell^2_+(\BC)$.
The observability Gramian $Q_\Phi$ determined
by the pair $\{C_\Phi, A\}$ is also given by
$Q_\Phi= W_{obs}^*T_R^{-1}W_{obs}$, the matrix  $R_0-\ga^*Q_\Phi\ga$ is strictly positive, and
\begin{equation}
\label{posCPhi}
C_\Phi=\Phi(0)(R_0-\ga^*Q_\Phi\ga)^{-1}(C-\ga^*Q_\Phi A).
\end{equation}
Finally, in this case, we may assume without loss of generality that $\Phi(0)$ is given by
\begin{equation}
\label{posPhi0}
\Phi(0)=(R_0-\ga^*Q_\Phi\ga)^{1/2}.
\end{equation}
\end{prop}

\begin{proof}[\bf Proof]
Since   $R(\z)$  is strictly positive for each $\z\in \BT$, the outer factor $\Phi$ is an invertible outer  {factor}, which is equivalent to   $T_\Phi$ being invertible. In particular,  $T_\Phi^*$ is surjective.  Thus for each $x\in \BC^n$ the vector $W_{obs}x$ belongs to $\im  T_\Phi^*$. This shows that the space $\sX_\Phi$ is equal to the full space $\BC^n$.  The two other identities in \eqref{invertout} then follow from \eqref{basicid}. Next, one computes that
\begin{align*}
Q_\Phi&= W_{\Phi, \,obs}^*W_{\Phi, \,obs}=W_{obs}^*T_\Phi^{-1}T_\Phi^{-*}W_{obs}  \\
    &=  W_{obs}^*(T_\Phi^{*}T_\Phi)^{-1}W_{obs}=W_{obs}^*T_R^{-1}W_{obs}.
\end{align*}
This proves    $Q_\Phi=W_{obs}^*T_R^{-1}W_{obs}$. Since $r=m$ and $\Phi(0)$ is invertible,  the identity    \eqref{r0ss}  shows that $R_0-\ga^*Q_\Phi\ga$ is strictly positive. Similarly, the identity \eqref{cstar}
{with \eqref{posPhi0}} yields \eqref{posCPhi}.

It remains to prove the final statement.  From \eqref{r0ss} and the fact that $\Phi(0)$ is
 invertible it follows that the polar decomposition of $\Phi(0)$ is
 given by  {$\Phi(0)=U(R_0-\ga^* Q_\Phi \ga)^{1/2}$}, where $U$ is unitary.
Recall that $\Phi$ is uniquely determined up to a unitary matrix from the left.
Thus without loss of generality we may replace $\Phi$ by $U^{-1}\Phi$, and then \eqref{posPhi0} holds.
\end{proof}

The results listed in the above proposition  also follow from Theorem 1.1. in  \cite{FKR1}; cf., Section 3 in \cite{FKR2a-10}. To be more specific let $R$ be as in   \eqref{defR}, and consider the associate  algebraic Riccati equation
\begin{equation}\label{are}
Q =  A^* Q A + ( C - \Gamma^* Q  A )^*
( R_0 - \Gamma^* Q\Gamma )^{-1} ( C - \Gamma^* QA ).
\end{equation}
An $n\ts n$ matrix $Q$  is called a  \emph{stabilizing solution} to  this algebraic Riccati equation if
\begin{itemize}
\item[\textup{(a)}]  $Q$ is a solution to \eqref{are},
\item[\textup{(b)}] $R_0 - \Gamma^* Q \Gamma$  is strictly positive,
\item[\textup{(c)}] the matrix $A-\ga( R_0 - \Gamma^* Q \Gamma )^{-1}( C - \Gamma^* Q A )$ is stable.
\end{itemize}
 It turns out that if the algebraic Riccati equation \eqref{are} admits a stabilizing  solution $Q$, then this solution is nonnegative  and unique. By the symmetric version of Theorem 1.1 in \cite{FKR1} (see Section 14.7 in \cite{BGKR2} or Sections 10.2 and 10.2 in \cite{FB10}) we know that the following are equivalent:
\begin{itemize}
\item[(i)]  The  values of the function $R$  on $\BT$ are strictly positive.
\item[(ii)] The function $R$ admits an invertible outer spectral factor $\Phi$, i.e., the outer spectral factor $\Phi$ is square and $T_\Phi$ is invertible.
\item[(iii)] The algebraic Riccati equation \eqref{are} admits a stabilizing solution $Q$.
\end{itemize}
Moreover, in this case, the following holds:
\begin{itemize}
\item[(1)] The invertible outer spectral  {factor} $\Phi$ of $R$ is given by
\begin{align}
&\Phi(z) = \Phi(0)+zC_0(I_n - z A)^{-1}\ga, \mbox{where} \label{outRpos}\\
& \hspace{1cm}\Phi(0) =( R_0 - \ga^* Q \ga ) ^{1/2},\nonumber\\
& \hspace{1.35cm}C_0=\Phi(0)( R_0 - \ga^* Q \ga )^{-1} ( C - \ga^* Q A ). \nonumber
\end{align}
and
\begin{equation}\label{invphi}
{\Phi(z)^{-1}=\Phi(0)^{-1}-z \Phi(0)^{-1}C_0 (I-zA^\ts  )^{-1}\ga \Phi(0)^{-1},}
\end{equation}
where $A^\ts=A-\ga( R_0 - \ga^* Q \Gamma )^{-1}( C - \ga^* Q A )$ is stable.
\item[(2)] The  unique stabilizing solution $Q$ to \eqref{are}  is given by
\begin{equation}\label{QTRinv}
Q = W_{obs}^* T_R^{-1} W_{obs}.
\end{equation}
\end{itemize}
Finally, if  in addition   $\{C,A\}$ is observable, then $W_{obs}$ is one to one, and thus, $Q$ is strictly positive.

From Proposition \ref{propinvTPhi} above we know that $Q_\Phi= W_{obs}^*T_R^{-1}W_{obs}$. But then   \eqref{QTRinv} shows that the stabilizing solution $Q$ of the Riccati  equation  \eqref{are} coincides with the observability Gramian $Q_\Phi$. Furthermore, $C_\Phi=C_0$  and the outer
 {spectral} factor  $\Phi$  in Proposition \ref{propinvTPhi} is equal to the   outer  {spectral} factor $\Phi$ given by \eqref{outRpos}. Finally, assuming $\{C,A\}$ is observable and  using the first identity in \eqref{invertout}, we conclude from \eqref{defOm} that $\om_\Phi=Q_\Phi$, and hence \eqref{proj1} tells us that  ${\de} = Q^{-1}_\Phi$.

Applied to the Leech problem \eqref{Leech1} the above results yield the following algorithm to compute a solution when $R$ admits an invertible outer spectral factor. This algorithm can be easily programmed in Matlab.

\begin{procedure}\label{thmpos1}
Let  $G\in {\fR}H^\iy_{m\ts p}$ and $K\in {\fR}H^\iy_{m\ts q}$  be given by the minimal realization \eqref{reprGK1}.   Consider the algebraic Riccati equation \eqref{are} where $R_0$ and $\ga$ are now given by \eqref{defR0} and \eqref{defga}, respectively.
\begin{itemize}
\item[\textup{(i)}]
Assume that there exists a stabilizing solution $Q$ to the algebraic Riccati equation \eqref{are}, or equivalently, the values of $R$ on the unit circle are strictly positive.

\item[\textup{(ii)}]
Then  there exists a stable rational matrix solution $X$ to the Leech problem \eqref{Leech1}   if and only if $Q^{-1} \geq  P_1-P_2$.
Therefore in what follows we assume that $Q^{-1} \geq  P_1-P_2$.
\end{itemize}
If \textup{(i)} and \textup{(ii)} hold, then such a solution $X$ can be computed  by the following steps:
 \begin{itemize}

 \item
 Let $\Phi(0)$ and $C_\Phi$ be the matrices defined by
 \[
 \Phi(0)=(R_0-\ga^*Q \ga)^{1/2} \  \mbox{and}\  C_\Phi=\Phi(0)(R_0-\ga^*Q\ga)^{-1}(C-\ga^*Q A).
 \]
   \item Find  matrices $B_\Theta$ and $D_\Theta$ such that
   \[
   \begin{bmatrix}
  A^*    &   C_\Phi^* \\
  B_\Theta^*    &   D_\Phi^*
\end{bmatrix}
\begin{bmatrix}
 Q    &  0 \\
  0    &  I_r
\end{bmatrix}
\begin{bmatrix}
  A      &   B_\Theta  \\
  C_\Phi   &   D_\Theta
\end{bmatrix}=
\begin{bmatrix}
 Q   &  0 \\
  0    &  I_r
\end{bmatrix}.
   \]

\item
Set $P_3 = Q^{-1}$ and $B_3 = B_\Theta$,  and put
\[
D_3 = \Phi(0)^*D_\Theta + \Gamma^* Q B_\Theta.
\]

\item
Use \textup{Theorem \ref{mainthm2}} to compute $U$ in \eqref{intropartiso1}. Then a
stable rational matrix solution {$X$} to \eqref{Leech1}  is given by $X(z) =  \delta_1 + z  \gamma (I - z \alpha)^{-1} \beta_1$, as in \eqref{introX}.

\item
The function $F(z) = D_3+zC(I_n-A)^{-1}B_3$ satisfies items
\textup{(i)} and \textup{(ii)} of \textup{Theorem \ref{mainthm1}}.

\item
Finally, $\Psi(z) =  \delta_2 + z  \gamma (I - z \alpha)^{-1} \beta_2$ is
a stable rational matrix   function    satisfying $G \Psi = F$  and  $\|\Psi\|_\infty \leq 1$;
see \textup{Theorem \ref{mainthm2}}.
 \end{itemize}
 \end{procedure}

\setcounter{equation}{0}
\section{Example}\label{secexample}
To gain some further insight into the solution obtained by the algorithm described by Procedure \ref{thmpos1}, let us  consider the simple case  when
\begin{equation}
\label{exGK}
G(z) = \frac{1}{\sqrt{2}}\begin{bmatrix}
         1  & 1 \\
       \end{bmatrix} \quad \mbox{and}\quad K(z) = \frac{z}{2}.
\end{equation}
Let $\tau$ be any function in $H^\infty$ satisfying  $\|\tau\|_\infty \leq 1$. One can easily see that
\begin{equation}
\label{exallsol}
X(z) = \frac{z}{2\sqrt{2}}\begin{bmatrix}
        1 \\
        1 \\
      \end{bmatrix} + \frac{ \sqrt{3}}{2\sqrt{2}} \begin{bmatrix}
        1 \\
        -1\\
      \end{bmatrix}\tau(z), \quad |z|<1,
\end{equation}
is a solution to the corresponding Leech problem \eqref{Leech1}. In fact,  all possible solutions are obtained in this way. Note that   the problem has  infinitely many stable rational solutions.

Here we will see that our algorithm yields the particular solution $X$ in \eqref{exallsol} with $\tau$ identically equal to zero, that is,
\begin{equation}
\label{exsol0}
X (z)= \frac{z}{2\sqrt{2}}\begin{bmatrix} 1 \\ 1 \\ \end{bmatrix}.
\end{equation}
(It turns out that  $X$ in \eqref{exsol0} is also  the minimal $H^\infty$ and the minimal $H^2$ solution to $G X =K$.)  For $G$ and $K$ in \eqref{exGK}, a state space realization for $\begin{bmatrix}
         G  & K
       \end{bmatrix}$ is given by \eqref{reprGK1} where
\begin{equation}\label{exdata}
A = 0, \,\,\, C = 1, \,\,\, B_1 = \begin{bmatrix}
         0  & 0 \\
       \end{bmatrix}, \,\,\,  D_1 = \frac{1}{\sqrt{2}}\begin{bmatrix}
         1  & 1 \\
       \end{bmatrix},\,\,\, B_2 = \frac{1}{2},\,\,\, D_2=0.
\end{equation}
With this choice the realization of  $\begin{bmatrix} G  & K  \end{bmatrix}$ is minimal. The controllability Gramians in \eqref{WobsP12} are given by $P_1=0$ and $P_2=1/4$. The function $R$ in \eqref{defRintro} is defined by  $R(z) = G(z)G(1/\bar{z})^* -  K(z)K(1/\bar{z})^* = 3/4$, and $\Gamma =0$. Hence  $\Phi(z)=\sqrt{3}/2$, the subspace $\sX_\Phi = \mathbb{C}$, and $C_\Phi =2/\sqrt{3}$; see \eqref{cstar}. The inner function $\tht$ is given by $\Theta(z) = z$, and $B_3= B_\Theta =\sqrt{3}/2$,  while  $D_\Theta =0$.
 Moreover, $F(z) = \Phi(1/\bar{z})^*\Theta(z) = z\sqrt{3}/2$
 and $D_3= 0$. The controllability Gramian $P_3$ of the pair $\{A,B_3\}$ is given by $P_3=3/4$. Therefore
 $P_3 + P_2 - P_1 = 1$ and $\Upsilon =1$.  According to item (i) in
 Theorem \ref{mainthm1}, the operator $T_G T_G^* - T_K T_K^*$ is
 nonnegative.

 Now we can use \eqref{defU1} to compute a contractive solution to $G X = K$. In this case, the observability Gramian $Y$ for the pair $\{C,A\}$ is given by  $Y=1$,  and $U = V^+V_1$ where
 \[
 V = \begin{bmatrix}
       1 & 0 & 0 \\[.2cm]
       0 & \frac{1}{2} & \frac{1}{2}\\[.2cm]
       0 & \frac{1}{2} & \frac{1}{2} \\[.2cm]
     \end{bmatrix}\mbox{ on } \mathbb{C}^3
     \quad \mbox{and} \quad
 V_1 =  \begin{bmatrix}
       0  & \frac{1}{2} & \frac{\sqrt{3}}{2}\\[.2cm]
       \frac{1}{\sqrt{2}} & 0 & 0\\[.2cm]
      \frac{1}{\sqrt{2}} & 0 & 0\\[.2cm]
     \end{bmatrix}\mbox{ on } \mathbb{C}^3.
 \]
Note that  $V$ is an orthogonal projection, and thus  $V^+ = V$.
 A simple calculation shows that
 \[
 U = V^+V_1 = V V_1=\begin{bmatrix}
       0  & \frac{1}{2} & \frac{\sqrt{3}}{2}\\[.2cm]
       \frac{1}{\sqrt{2}} & 0 & 0\\[.2cm]
       \frac{1}{\sqrt{2}} & 0 & 0 \\[.2cm]
     \end{bmatrix}.
 \]
 Hence
 \[
 \alpha = 0,\quad \gamma = \begin{bmatrix}
        \frac{1}{\sqrt{2}}\\[.2cm]
        \frac{1}{\sqrt{2}}\\[.2cm]
     \end{bmatrix}, \quad \beta_1 = \frac{1}{2},\quad \beta_2 = \frac{\sqrt{3}}{2},
\quad  \delta_1 = \begin{bmatrix}
       0  \\[.2cm]
       0   \\[.2cm]
     \end{bmatrix}, \quad     \delta_2 = \begin{bmatrix}
       0 \\[.2cm]
       0  \\[.2cm]
     \end{bmatrix}.
 \]
 Therefore
\[
X(z) = \delta_1 + z \gamma(I_1 - z \alpha)^{-1} \beta_1 = \frac{z}{2\sqrt{2}}\begin{bmatrix}
       1 \\ 1 \\
     \end{bmatrix}
     \]
is a stable rational matrix solution to the Leech problem \eqref{Leech1} with $G$ and $K$ as in \eqref {exGK}.  Finally,
\[
\Psi(z) = \delta_2 + z \gamma(I_1 - z \alpha)^{-1} \beta_2 = \frac{z \sqrt{3}}{2\sqrt{2}}\begin{bmatrix} 1 \\ 1 \\   \end{bmatrix}
\]
is a contractive stable rational matrix solution to $G \Psi = F$.

\begin{remark} The example presented in this section is of a special kind. Recall  that $R(z)\equiv3/4$, and thus $R$ is strictly positive  on $\BT$.  Hence in constructing
a rational solution to the Leech problem  we could have used   the procedure described in Procedure \ref{thmpos1} to get the solution $X$.  Note that in this case, given the data \eqref{exdata} and the equalities $R_0=3/4$ and $\ga=0$,  the Riccati equation \eqref{are} reduces to $Q=4/3$. The procedure outlined in Procedure \ref{thmpos1} then yields the same solution $X$ as the one obtained above.

Another special feature of the above example is the fact that $P_2-P_1=1/4$ is positive. This implies that for any stable  rational function $F$ such that $F(z)F(\bar{z}^{-1})^*=R(z)=3/4$,  not only the one constructed above,   the operator $T_GT_G^*-T_KT_K^* -T_FT_F^*$ is non-negative. This fact follows from the following variant of \eqref{fundid3b}:
\[
T_G T_G^* - T_K T_K^*-T_FT_F^* =
H_F H_F^*+W_{obs}\big(P_2-P_1\big)W_{obs}^*.
\]

\end{remark}



\begin{thebibliography}{xx}


\bibitem{Ando90}
T. Ando, \emph{De Branges spaces and analytic operator functions}, Sapporo, Japan, 1990.

\bibitem{A50}
N. Aronszajn, Theory of reproducing kernels, {\em Trans.\ Amer.\ Math.\ Soc.} \textbf{68} (1950), 337--404.

\bibitem{BT98}
J.A. Ball and T.T. Trent, Unitary colligations, reproducing kernel Hilbert spaces,
and Nevanlinna-Pick interpolation in several variables, {\em J. Funct.\ Anal.} {\bf 157} (1998), 1-–61.

\bibitem{BGKR08}
H. Bart, I. Gohberg, M.A. Kaashoek, and A.C.M. Ran, \emph{Factorization of matrix and operator functions: the state space method}, Oper.\ Theory Adv.\ Appl.\  \textbf{178}, Birkh\"auser Verlag, Basel, 2008.

\bibitem{BGKR2}
H. Bart, I. Gohberg, M.A. Kaashoek, and A.C.M. Ran, \emph{A state space approach to canonical factorization: convolution equations and mathematical systems}, Oper.\ Theory Adv.\ Appl.\  \textbf{200}, Birkh\"auser Verlag, Basel, 2010.

\bibitem{BS90}
A. B\"ottcher  and B. Silbermann, \emph{Analysis of Toeplitz Operators}, (Akademie-Verlag) Springer-Verlag, 1990.


\bibitem{Carl62}
L. Carlson, Interpolation by bounded analytic functions and the corona problem, \emph{Ann.\ Math.} \textbf{76} (1962), 547--559.

\bibitem{CF03}
M.J. Corless and A.E. Frazho, \emph{Linear sytems and control}, Marcel Dekker, Inc.,
New York, 2003.

\bibitem{DSS71}
R.G. Douglas, H.S. Shapiro, and A.L. Shields, Cyclic vectors and invariant subspaces for the backward shift, \emph{Ann. Inst. Fourier Grenoble} \textbf{20} (1971), 37--76.


\bibitem{FF90}
C. Foias  and A.E. Frazho, \emph{The Commutant Lifting Approach to Interpolation Problems}, Oper.\ Theory Adv.\ Appl.\  \textbf{44}, Birkh\"auser Verlag, Basel, 1990.

\bibitem{FFGK98}
C. Foias, A.E. Frazho, I. Gohberg, and M.A. Kaashoek, \emph{Metric constrained interpolation, commutant lifting and systems}, Oper.\ Theory Adv.\ Appl.\  \textbf{100}, Birkh\"auser Verlag, Basel, 1998.

\bibitem{FB10}
A.E. Frazho and W. Bosri, \emph{An operator perspective on signals and systems},  Oper.\ Theory Adv.\ Appl.\  \textbf{204}, Birkh\"auser Verlag, Basel, 2010.

\bibitem{FKR1}
A.E. Frazho, M.A. Kaashoek, and A.C.M. Ran, The non-symmetric discrete algebraic Riccati equation and canonical factorization of rational matrix functions on the unit circle, \emph{Integr.\ Equ.\ Oper.\ Theory} \textbf{66} (2010), 215--229.

\bibitem{FKR2a-10}
A.E. Frazho, M.A. Kaashoek, and A.C.M. Ran, Right invertible multiplication operators and $H^2$ solutions to a rational Bezout equation, I. Least squares solution, \emph{Integr.\ Equ.\ Oper.\ Theory} \textbf{70} (2011), 395--418.

\bibitem{Fuhr68}
P. Fuhrmann, On the corona theorem and its applications to spectral problems in Hilbert space, \emph{Trans.\ Amer.\ Math.\ Soc.} \textbf{132} (1968), 55--66.

\bibitem{GGK2}
I. Gohberg, S. Goldberg, and  M.A. Kaashoek, \emph{Classes of Linear Operators}, Volume I, Oper.\ Theory Adv.\ Appl.\  \textbf{63}, Birkh\"auser Verlag, Basel, 1993.

\bibitem{tH13}
S. ter Horst, Rational matrix solutions to the Leech equation: The Ball-Trent approach revisited, submitted.

\bibitem{KZ99}
M.A. Kaashoek and C.G. Zeinstra, The band method and generalized Carathéodory-Toeplitz interpolation at operator points, \emph{Integr.\ Equ.\ Oper.\ Theory} \textbf{33} (1999), no.\ 2, 175–-210.

\bibitem{RR85}
M. Rosenblum and J. Rovnyak, {\em Hardy classes and operator theory}, Oxford Mathematical Monographs, Oxford Science Publications, The Clarendon Press, Oxford University Press, New York, 1985.

\bibitem{NFBK}
B. Sz.-Nagy, C. Foias, H. Bercovici and L. K\'{e}rchy, \emph{Harmonic analysis of operators on Hilbert space}, Springer, New York, 2009.

\bibitem{Trent07}
T.T. Trent, An algorithm for the corona solutions on $H^\iy(D)$, \emph{Integr.\ Equ.\ Oper.\ Theory} \textbf{59} (2007), 421--435.

\bibitem{Trent12}
T.T. Trent, A Constructive Proof of the Leech Theorem for Rational Matrix Functions, \emph{Integr.\ Equ.\ Oper.\ Theory} \textbf{75} (2013), 39--48.

\bibitem{WB12}
S. Wahls, and H. Boche, Lower bounds on the infima in some $\sH_\infty$ optimization problems, {\em IEEE transactions  on automatic control} {\bf 57} (2012), 788--793.

\bibitem{WBP09}
S. Wahls, H. Boche, and V. Pohl, Zero-forcing precoding for frequency selective MIMO channels with $H^\infty$ criterion and causality constraint, \emph{Signal Processing} {\bf 89} (2009), 1754--1761.

\end{thebibliography}
\end{document}